\documentclass[reqno]{amsart}

\usepackage[T1]{fontenc}
\usepackage[utf8]{inputenc}
\usepackage[english]{babel}
\usepackage{times}
\usepackage{graphicx}
\usepackage{float}
\usepackage{amsmath,amssymb,amsfonts}
\usepackage{amsthm}
\usepackage{booktabs}
\usepackage{mathtools}
\usepackage{enumitem}
\usepackage{longtable}
\usepackage{array}
\usepackage{tabularx}
\usepackage{tikz}
\usepackage[top=1.5in,bottom=1.43in,left=1.25in,right=1.25in]{geometry}
\usepackage{hyperref}

\renewcommand{\leq}{\leqslant}

\renewcommand{\geq}{\geqslant}

\allowdisplaybreaks
\newtheorem{theorem}{Theorem}[section]
\newtheorem{lemma}[theorem]{Lemma}
\newtheorem{proposition}[theorem]{Proposition}
\newtheorem{corollary}[theorem]{Corollary}

\newtheorem{exactlemma}[theorem]{Exact finite lemma}
\newtheorem{conjecture}[theorem]{Conjecture}

\theoremstyle{definition}
\newtheorem{example}[theorem]{Example}
\newtheorem{definition}[theorem]{Definition}

\theoremstyle{remark}

\hypersetup{
  hidelinks,
  hypertexnames=false,
  bookmarksdepth=subsubsection,
  pdflang={en-US},
  pdfdisplaydoctitle=true,
  pdftitle={Prescribed-Difference Matchings with Four and Eight Holes: Fourier Filters for Compatible Boundaries},
  pdfauthor={Avraham Kreindel, Aryeh Lev Zabokritskiy (Yohananov)},
  pdfkeywords={prescribed differences, perfect matchings, Hall theorem, Walsh-Hadamard transform, complementary minors}
}

\newcommand{\F}{\mathbb F}
\newcommand{\wt}{\operatorname{wt}}
\newcommand{\per}{\operatorname{per}}
\newcommand{\vTwo}{\nu_2}
\newcommand{\eps}{\varepsilon}
\newcommand{\whK}{\widehat K}

\newcommand{\cD}{\mathcal D}
\newcommand{\cP}{\mathcal P}

\newcommand{\Span}{\operatorname{span}}

\begin{document}

\title[Fourier Filters for Compatible Boundaries]{Prescribed-Difference
Matchings with Four and Eight Holes: Fourier Filters for Compatible Boundaries}

\author[A. Kreindel]{Avraham Kreindel}
\address{Department of Computer Science, Reichman University, Herzliya, Israel}
\email{avrahamkreindel@gmail.com}

\author[A. L. Zabokritskiy (Yohananov)]{Aryeh Lev Zabokritskiy (Yohananov)}
\address{Department of Computer Science, MIGAL -- Galilee Research Institute,
Kiryat Shmona, Israel}
\address{Department of Computer Science, Tel-Hai University of Kiryat Shmona
and the Galilee, Kiryat Shmona, Israel}
\email{yuhanalev@telhai.ac.il}

\begin{abstract}
Let $s\geq2$ and let
$v_1,\ldots,v_{2^{s-1}}\in\F_2^s\setminus\{0\}$ have sum zero.  The
prescribed-difference matching problem asks whether $\F_2^s$ can be
partitioned into pairs whose differences, counted with multiplicity, are
exactly these vectors.  Fix a hyperplane $H\leq\F_2^s$.  An $h$-hole instance
is one in which exactly $h$ prescribed differences lie in $H$, counted with
multiplicity.  We prove that every four-hole instance has a solution for
$s\geq3$ and every eight-hole instance has a solution for $s\geq4$.  Thus no
restriction on the multiplicities or on the sum of the holes is needed beyond
the global zero-sum condition.

After the internal pairs are placed, their endpoints are deleted from the two
affine halves determined by $H$, and the remaining vectors must be paired with
the prescribed crossing differences.  A locally valid placement of the
internal pairs need not admit such a completion.  We encode the possible
completions by coefficients of signed determinants and use the Walsh
transform to sum over all disjoint legal placements while keeping the
crossing profile fixed.  Nonvanishing of this sum guarantees that at least one
placement extends to a full matching.

The four-hole and zero-sum eight-hole theorems are proved theoretically.  The
general eight-hole proof uses three finite exact verifications, independent of
the crossing profile: a local filter calculation, a classification of the
remaining eight-hole configurations, and integer certificates for those
configurations.  The verification software, exact inputs, and recorded outputs
are archived in a versioned Zenodo record.
\end{abstract}



\maketitle

\section{Introduction}\label{sec:intro}

Let $s\geq2$, set $V=\F_2^s$, and put $N=2^{s-1}$.  A
\emph{request instance} is a multiset
$\cD=\{v_1,\ldots,v_N\}$ of nonzero vectors in $V$; it is
\emph{admissible} if $\sum_i v_i=0$.  The prescribed-difference matching
problem asks for a
partition of $V$ into pairs
\[
  V=\bigsqcup_{i=1}^{N}\{x_i,y_i\}
  \qquad\text{such that}\qquad
  x_i+y_i=v_i \quad (i=1,\ldots,N).
\]
We use the following equivalent graph model throughout.  Its vertices are the
elements of $V$, every unordered pair $\{x,y\}$ is an edge, and that edge is
labeled by $x+y$.  Thus the graph is complete, and a solution is a perfect
matching whose edge labels, counted with multiplicity, are
$v_1,\ldots,v_N$.
Admissibility is necessary.  The conjecture of Balister, Gy\H{o}ri, and
Schelp asserts that it is also sufficient
~\cite{BalisterGyoriSchelp2011}.  Batch codes, introduced
in~\cite{IshaiEtAl2004}, give a coding-theoretic interpretation in the binary
functional setting: a prescribed-difference perfect matching yields
$2^{s-1}$ disjoint recovery sets from the $2^s-1$ nonzero Simplex columns
(after the formal zero column is deleted, its incident edge becomes a
singleton recovery set).
Consequently, BGS would make the nonzero Simplex columns an
$\mathrm{FB}\text{-}(2^s-1,s,2^{s-1})$ binary linear primitive
functional multiset batch code whose recovery sets have size at most two.
Writing $\mathrm{FB}(s,k)$ and $\mathrm{FP}(s,k)$ for the minimum lengths of
binary functional batch and functional PIR codes, respectively, its length is
optimal because
\[
 \mathrm{FB}(s,2^{s-1})\geq \mathrm{FP}(s,2^{s-1})=2^s-1
\]
by~\cite[Theorems~2(6) and~4(1)]{ZhangEtzionYaakobi2020}; the matching
construction attains the lower bound.  Additional sufficient conditions for
this optimal batch-code regime were developed in
~\cite{HollmannEtAl2023,YohananovYaakobi2022,YohananovEssayag2025}.

The hyperplane framework in~\cite{TwoHoleManuscript} organizes this problem
by fixing a hyperplane $H\leq V$ and writing
\[
  V=H\sqcup(\alpha+H),\qquad \alpha\notin H.
\]
A prescribed difference outside $H$ must be realized by an edge crossing
between the two affine halves; a difference in $H$ must be realized by an
internal edge contained in one half.  If exactly $h$ prescribed differences
belong to $H$, we call the instance an $h$-hole instance.  In every solution,
$h/2$ of the corresponding internal edges lie in each affine half.  Their
endpoints occupy $h$ vertices on each side, leaving a crossing problem with
$h$ deleted vertices in each copy of $H$.  The hierarchy $h=0,2,4,8$ is shown
in Figure~\ref{fig:puncture-hierarchy}.

\begin{theorem}[Main theorem]\label{thm:main}
Let $H\leq\F_2^s$ be a hyperplane and let $\cD$ be an admissible request
instance.
\begin{enumerate}[label=\textup{(\roman*)}]
\item If $s\geq3$ and exactly four requests of $\cD$, counted with
multiplicity, lie in $H$, then $\cD$ has a solution.
\item If $s\geq4$ and exactly eight requests of $\cD$, counted with
multiplicity, lie in $H$, then $\cD$ has a solution.
\end{enumerate}
\end{theorem}

The theorem is existential.  The nonvanishing argument certifies that at
least one compatible boundary exists, but it does not provide an extraction
algorithm or a complexity bound.  By contrast, the earlier even-multiplicity
four-hole result in~\cite{TwoHoleManuscript} is constructive.

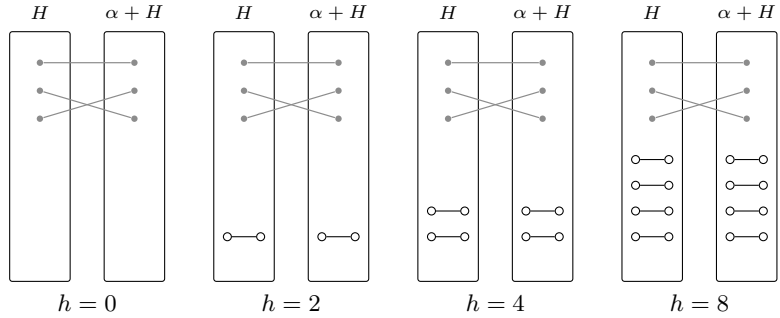
\begin{figure}[H]
\centering
\begin{tikzpicture}[
  x=1cm,y=1cm,
  half/.style={draw,rounded corners=1pt,minimum width=.8cm,minimum height=3.3cm},
  hole/.style={circle,draw,fill=white,inner sep=0pt,minimum size=3pt},
  cross/.style={circle,fill=black!45,inner sep=0pt,minimum size=2.4pt},
  crossedge/.style={black!45,line width=.45pt}
]
\foreach \shift/\pairs/\holes in {0/0/0,2.7/1/2,5.4/2/4,8.1/4/8}{
  \begin{scope}[xshift=\shift cm]
    \node[half] (L) at (0,1.75) {};
    \node[half] (R) at (1.25,1.75) {};
    \node[font=\scriptsize] at (0,3.65) {$H$};
    \node[font=\scriptsize] at (1.25,3.65) {$\alpha+H$};
    \node[font=\small] at (.625,-.18) {$h=\holes$};
    \node[cross] (lc1) at (0,2.25) {};
    \node[cross] (lc2) at (0,2.62) {};
    \node[cross] (lc3) at (0,2.99) {};
    \node[cross] (rc1) at (1.25,2.62) {};
    \node[cross] (rc2) at (1.25,2.25) {};
    \node[cross] (rc3) at (1.25,2.99) {};
    \draw[crossedge] (lc1)--(rc1);
    \draw[crossedge] (lc2)--(rc2);
    \draw[crossedge] (lc3)--(rc3);
    \ifnum\pairs>0
      \foreach \j in {1,...,\pairs}{
        \pgfmathsetmacro{\yy}{.35+.34*\j}
        \node[hole] (la\j) at (-.22,\yy) {};
        \node[hole] (lb\j) at (.22,\yy) {};
        \draw (la\j)--(lb\j);
        \node[hole] (ra\j) at (1.03,\yy) {};
        \node[hole] (rb\j) at (1.47,\yy) {};
        \draw (ra\j)--(rb\j);
      }
    \fi
  \end{scope}
}
\end{tikzpicture}
\caption{The puncture hierarchy relative to
$V=H\sqcup(\alpha+H)$.  Hollow endpoints are occupied by internal edges;
three representative crossing edges are shown in gray.  From left to right,
$h=0,2,4,8$ correspond to Hall's problem, the two-hole level, and the two
levels treated here.}
\label{fig:puncture-hierarchy}
\end{figure}

The case $h=0$ is Hall's permutation problem~\cite{Hall1952}: every vertex of
one copy of $H$ is matched to a vertex of the other copy.  Hall's technique
does not, however, solve the first punctured case.  Already for $h=2$, one
internal edge must be placed in each affine half, and its two endpoints are
not known in advance.  Hall's argument neither selects the two deleted
vertices on each side nor guarantees that an arbitrary locally legal
selection supports the prescribed crossing differences.  It therefore does
not settle even two holes, and a fortiori it does not settle four or eight
holes.

The punctured problems and their common boundary-selection obstruction were
formulated in~\cite{TwoHoleManuscript}.  The two-hole problem was solved
there without a multiplicity assumption.  The same analysis gave an
algorithmic four-hole construction under the stronger hypothesis that every
request value has even multiplicity.  The present paper advances this
hierarchy in two directions.  First, it proves the four-hole theorem with no
multiplicity restriction: the four internal directions and the crossing
profile are arbitrary, subject only to the zero-sum condition.  Second, it
proves the eight-hole theorem.  When the eight internal differences sum to
zero, the argument is theoretical; the general statement adds three finite
exact verifications independent of the crossing profile, archived on
Zenodo~\cite{KreindelZabokritskiyVerification2026}.  Balister,
Gy\H{o}ri, and Schelp report that the full conjecture holds for
$s\leq5$~\cite{BalisterGyoriSchelp2011}.  Since that assertion is not
accompanied there by an inspectable proof or finite certificate, the proofs
below cover these dimensions internally and use the assertion only for
historical comparison.  Relative to the reported range, the new dimensional
content begins at $s\geq6$.

The new difficulty is not simply the number of internal edges.  Their
prescribed directions must be divided between the two affine halves, their
endpoints must be chosen without collisions, and the remaining vertices must
still support the prescribed crossing differences.  These choices are
coupled: a locally legal placement of the internal edges need not extend to a
full matching.

Our proof has two layers.  First, a coefficient of a signed determinant tests
whether a fixed choice of deleted endpoints admits the required crossing
matching.  Second, a Fourier sum ranges over all locally legal choices of the
internal edges and proves that at least one of those determinant coefficients
is nonzero.  For four holes, the possible deleted sets split into two
geometric types, each governed by a different nonvanishing law.  For eight
holes, the zero-sum case admits a uniform theoretical argument; the general
case requires an additional reduction and the same three Zenodo-archived
verifications.
All algebraic objects used in this summary are defined in
Section~\ref{sec:framework}.

Related pairing problems for cyclic groups of odd prime order and for finite
fields of odd characteristic were studied by Preissmann and Mischler and by
Karasev and Petrov, respectively
~\cite{PreissmannMischler2009,KarasevPetrov2012}.  Within the binary setting,
Correia, Pokrovskiy, and Sudakov proved a broad near-perfect rainbow-matching
result that realizes arbitrary prescribed differences up to a sublinear
defect, but does not reach a perfect matching
~\cite{CorreiaPokrovskiySudakov2023}.  Other
earlier results use parameters different from the number of holes.
Balister, Gy\H{o}ri, and Schelp reported that the conjecture holds for
$s\leq5$; they proved it for a family in which one value occurs at least
$N/2$ times and each request value has even multiplicity
~\cite{BalisterGyoriSchelp2011}.  Golowich and Kim proved it whenever at most
$s$ distinct request values occur
~\cite{GolowichKim2020}.  Kov\'acs obtained a complementary multiplicity
theorem: for every $\varepsilon>0$, if $s$ is sufficiently large and one
value occurs in at least a $1/2+\varepsilon$ fraction of the requests, then a
perfect matching exists~\cite{Kovacs2026}.  The all-crossing hyperplane family
was obtained constructively by Yohananov and Yaakobi
~\cite{YohananovYaakobi2022} and also follows by adapting the Hall-based
argument of Hollmann et al.~\cite{HollmannEtAl2023}.  Erde gave a direct
hypercube proof for the coordinate-direction subcase, in which every request
is a standard basis vector~\cite{Erde2026}.  These hypotheses constrain the
dimension, the set of distinct request values, or the multiplicity profile of
the request multiset.

There is also a precise asymptotic intersection with the random Hall--Paige
theory of M\"uyesser and Pokrovskiy~\cite{MuyesserPokrovskiy2025}.  Put
$K:=H$.  Fix $h$ and choose $h/2$ pairwise vertex-disjoint internal edges in
each affine half; call such a choice a collision-free boundary.  After
translating the right half back to $K$, let $A,B\subseteq K$ be its two
endpoint sets, so $|A|=|B|=h$.
If every residual crossing difference occurs once, write $Z$ for their set.
Then $X=K\setminus A$, $Y=K\setminus B$, and $Z$ have the same size, and
admissibility gives
\[
  \sum X+\sum Y=\sum Z.
\]
To apply~\cite[Theorem~1.1]{MuyesserPokrovskiy2025}, take $p=1$ and the
three reference sets $R^1=R^2=R^3=K$.  The sets $X,Y,Z$ have the same
cardinality, the displayed identity is the required balance condition, and
\[
 |X\mathbin\triangle K|+|Y\mathbin\triangle K|
   +|Z\mathbin\triangle K|=3h.
\]
The lower bound on $p$ is automatic, and for fixed $h$ the perturbation
hypothesis
\[
 3h\leq \frac{|K|}{(\log |K|)^{10^{30}}}
\]
holds for all sufficiently large $|K|$.  The theorem therefore yields a
bijection $\phi:X\to Y$ for which $x\mapsto x+\phi(x)$ is a bijection from
$X$ to $Z$, and hence completes every such boundary.
This covers the squarefree residual regime asymptotically, but not the
arbitrary repeated-direction profiles, explicit finite-rank thresholds, or
deleted-minor certificates required here.  Our parameter is the number of
requests lying in a fixed hyperplane, while the main new difficulty is to
control all multiplicity profiles at the fixed levels $h=4$ and $h=8$.

\section{Preliminaries}
\label{sec:prelim}

\subsection{The problem and the hyperplane hierarchy}

All vector spaces are over $\F_2$.  We retain the notation introduced above:
an admissible request instance in $V=\F_2^s$ is
\[
  \cD=\{v_1,\ldots,v_N\},
  \qquad N=2^{s-1},
  \qquad \sum_i v_i=0.
\]
A \emph{solution} is a partition
\[
  V=\bigsqcup_{i=1}^{N}\{x_i,y_i\}
  \quad\text{such that}\quad x_i+y_i=v_i
  \quad(i=1,\ldots,N).
\]

\begin{conjecture}[Balister--Gy\H{o}ri--Schelp]
Every admissible request instance in $\F_2^s$ has a solution.
\end{conjecture}

Fix a hyperplane $H\leq V$ and $\alpha\in V\setminus H$, so that
\[
  V=H\sqcup(\alpha+H).
\]
A request in $\alpha+H$ is \emph{crossing}; a request in $H$ is
\emph{internal}.  We also call an internal request an \emph{$H$-hole},
counting multiplicity.  An \emph{$h$-hole instance} has exactly $h$
internal requests.

For an admissible instance, applying the quotient map $V\to V/H\cong\F_2$
to the zero-sum condition shows that $N-h$ is even.  In the cases considered
here $N$ is even, so $h$ is even.  In every solution, precisely $h/2$
internal edges lie in each affine half: the $N-h$ crossing edges use $N-h$
vertices of each half, leaving $h$ vertices in each half for internal edges.
Thus the crossing core is a bijection between two copies of $H$ with $h$
vertices deleted from each.  For $x\in\F_2^s$, let $\wt(x)$ denote its
Hamming weight, namely the number of nonzero coordinates.  For the parity
hyperplane
\[
  H_{\rm even}=\{x\in\F_2^s:\wt(x)\equiv0\pmod2\},
\]
the holes are exactly the even-weight requests.

The next two theorems are cited background results; they are not proved in
this paper.  We restate the exact forms used below.  The zero-hole statement
is the binary case of Hall's theorem~\cite{Hall1952} (see also the
simplex-code formulation in~\cite{HollmannEtAl2023}), and the two-hole
statement is proved in~\cite{TwoHoleManuscript}.

\begin{theorem}[Binary Hall theorem (cited)]\label{thm:hall-stated}
Every admissible zero-hole instance has a solution.  Equivalently, if
$K\cong\F_2^r$ and $d_1,\ldots,d_{|K|}\in K$ satisfy
$\sum_i d_i=0$, then there is a permutation $\pi:K\to K$ whose multiset
of differences $\{x+\pi(x):x\in K\}$ is $\{d_1,\ldots,d_{|K|}\}$.
\end{theorem}

\begin{theorem}[Two-hole theorem (cited)]\label{thm:two-hole-stated}
Every admissible two-hole instance has a solution, with no multiplicity
assumption on either the internal or the crossing requests.
\end{theorem}

\medskip
\noindent\textbf{Roadmap.}
Section~\ref{sec:prelim} states the basic matching problem, the hyperplane
hierarchy, and the cited zero- and two-hole results.  Section
~\ref{sec:framework} defines compatible boundaries, determinant certificates,
the Walsh transform, and the filtered-minor reduction.  Section
~\ref{sec:deleted-walsh-prep} proves the coefficient laws used in every later
case.  Sections~\ref{sec:four}, \ref{sec:eight-zero}, and
~\ref{sec:eight-nonzero} prove the four-hole, zero-sum eight-hole, and general
eight-hole statements, respectively.  Appendix~\ref{app:certificates}
records the finite certificates used in the last branch.  The exact
verification procedures, inputs, and recorded outputs are archived in the
Zenodo record~\cite{KreindelZabokritskiyVerification2026}.

\section{The character-minor framework}
\label{sec:framework}\label{sec:common-map}

\subsection{Boundary selection and determinant certificates}

The formal reduction is valid for any number $h$ of deleted vertices on each
side.  In a matching instance $h$ is even, because the boundary in each
affine half is the union of $h/2$ internal edges.  Four holes will be the
running case: two prescribed internal edges must be placed in each half,
producing four deleted rows and four deleted columns.

Put $r=s-1$, $K=H\cong\F_2^r$, and $N=|K|=2^r$.  Write each crossing request uniquely
as $\alpha+d_i$ with $d_i\in K$.  If there are $h$ internal requests, the
reduced crossing list has length $N-h$.

Suppose that the internal requests have been split into multisets $P_L$ and
$P_R$.  A pair $(A,B)$ of $h$-element subsets of $K$ is \emph{locally legal
for this split} if $A$ is the endpoint set of $h/2$ pairwise vertex-disjoint
edges in $K$ whose difference multiset is $P_L$, and $B$ is the endpoint set
of $h/2$ pairwise vertex-disjoint edges in the second copy of $K$ whose
difference multiset is $P_R$.  If the reduced crossing list is
$d_1,\ldots,d_{N-h}\in K$, a locally legal boundary $(A,B)$ is
\emph{crossing-compatible} with that list if there is a bijection
\[
  \pi:K\setminus A\longrightarrow K\setminus B
\]
such that
\[
  \{x+\pi(x):x\in K\setminus A\}
  =\{d_1,\ldots,d_{N-h}\}
\]
as multisets.  Thus crossing-compatibility is exactly the condition that the
two internal boundary matchings extend to a solution.  The following
elementary assembly step will also be used for eight holes.

\begin{lemma}[Boundary assembly]\label{lem:boundary-assembly}
Suppose the internal requests have been split into multisets $P_L$ and
$P_R$, each of size $h/2$.  Let $\mathcal E_L$ and $\mathcal E_R$ be
disjoint-edge realizations of these multisets in $K$, with endpoint sets $A$
and $B$, respectively; thus $|A|=|B|=h$.  If a bijection
\[
  \pi:K\setminus A\longrightarrow K\setminus B
\]
has reduced difference multiset
\[
  \{x+\pi(x):x\in K\setminus A\}
  =\{d_1,\ldots,d_{N-h}\},
\]
where both sides are multisets, then the edges in
$\mathcal E_L$, the translated edges
$\{\{\alpha+u,\alpha+v\}:\{u,v\}\in\mathcal E_R\}$, and the crossing edges
\[
  \{x,\alpha+\pi(x)\},\qquad x\in K\setminus A,
\]
form a solution for the original requests
$P_L$, $P_R$, and $\alpha+d_1,\ldots,\alpha+d_{N-h}$.
\end{lemma}

\begin{proof}
The left internal edges cover $A$, the translated right internal edges cover
$\alpha+B$, and the bijectivity of $\pi$ makes the crossing edges cover the
two complementary vertex sets.  Thus the three edge families are disjoint
and cover $V$.  The displayed crossing edge has difference
$\alpha+x+\pi(x)$, while translation by $\alpha$ does not change the
differences of the right internal edges.
\end{proof}

For the reduced crossing list, let $m_x$ be the multiplicity of $x\in K$ and
call $\mathbf m=(m_x)_{x\in K}$ its \emph{reduced profile}.  Write
\[
 |\mathbf m|=\sum_{x\in K}m_x,
 \qquad \sigma(\mathbf m)=\sum_{x\in K}m_x\,x,
 \qquad z^{\mathbf m}=\prod_{x\in K}z_x^{m_x},
 \qquad \mathbf m!=\prod_{x\in K}m_x!.
\]
The sum in $\sigma(\mathbf m)$ is taken in $K$ (so the integer $m_x$ acts
modulo two).  If the $h$ internal
directions are $h_1,\ldots,h_h$, admissibility gives
\[
  \sigma(\mathbf m)=h_1+\cdots+h_h.
\]

Let
\[
 \mathcal R_{\mathrm{pol}}=\mathbb Z[z_d:d\in K]
\]
be the polynomial ring in commuting indeterminates indexed by the reduced
differences.  Define the group circulant
$C(z)\in\operatorname{Mat}_N(\mathcal R_{\mathrm{pol}})$,
with rows and columns indexed by $K$, by
\[
  C(z)=(z_{x+y})_{x,y\in K}.
\]
Fix an arbitrary total order on $K$ for matrix rows and columns.  It has no
mathematical content; changing it can only change signs of determinants.
The rows and columns are possible endpoints.  For $X\subseteq K$, write
$X^c=K\setminus X$; the notation $C(z)_{A^c,B^c}$ means the submatrix with
row-index set $A^c$ and column-index set $B^c$.  If $P$ is a polynomial,
$[z^{\mathbf m}]P$ denotes the coefficient of the monomial
$z^{\mathbf m}$.  For a square matrix $M=(M_{u,v})$ with equally large row
and column index sets, its permanent is
\[
  \per M=\sum_{\pi}\prod_u M_{u,\pi(u)},
\]
where the sum is over all bijections from the row-index set to the
column-index set.  For any matrix $M$ and row and column index sets $I,J$,
write $M[I,J]$ for the corresponding submatrix; the permanent of the empty
$0\times0$ matrix is $1$.  With these conventions, the coefficient
\[
 [z^{\mathbf m}]\per C(z)_{A^c,B^c}
\]
is exactly the number of bijections $K\setminus A\to K\setminus B$ with
profile $\mathbf m$.  Thus positivity of this permanent coefficient is
equivalent to the existence of such a bijection.  We shall use the stronger
signed certificate
\begin{equation}\label{eq:minor-certificate}
 [z^{\mathbf m}]\det C(z)_{A^c,B^c}\neq0
 \quad\Longrightarrow\quad
 \begin{gathered}
 \text{there exists a bijection }K\setminus A\longrightarrow K\setminus B\\
 \text{with difference profile }\mathbf m.
 \end{gathered}
\end{equation}
Indeed, the determinant coefficient is the signed sum of the same
bijections.  Nonvanishing is sufficient, although not necessary, because
even and odd bijections may cancel.  The dependence on both $A,B$ and the
target monomial is essential: a locally legal boundary need not be
crossing-compatible, and a nonzero minor polynomial need not contain the
coefficient $z^{\mathbf m}$ required by the given crossing requests.

Thus the matching problem has been reduced to one algebraic search: among
the locally legal boundary placements, find $A,B$ for which the prescribed
coefficient of the punctured determinant is nonzero.  The rest of the common
map moves this search from endpoint space to character space and then averages
only over legal edge placements.  The later case sections will supply the final
noncancellation argument.

\subsection{Fourier diagonalization and boundary filters}
\label{sec:shared-filter}

The matrix $C(z)$ is difficult in the endpoint basis because its determinant
mixes all bijections at once.  It also has a strong symmetry: its $(x,y)$
entry depends only on $x+y$.  Fourier analysis is the change of basis adapted
to exactly this translation symmetry.  In the physical analogy, the points
of $K$ are locations and the characters are frequencies.  Here a frequency
is not a sine wave but a parity pattern on the binary cube: it assigns $+1$
to one affine hyperplane and $-1$ to the other.

Let
\[
  \whK=\operatorname{Hom}_{\F_2}(K,\F_2)
\]
be the dual space of additive characters with values in $\F_2$.  For
$\xi\in\whK$ and $x\in K$, the bit $\xi(x)\in\F_2$ is also denoted by
$\langle\xi,x\rangle$.  After coordinates are chosen, the standard
nondegenerate pairing identifies $\whK$ with $K$, but the constructions use
only this evaluation pairing.  Define the associated sign character by
\[
  \eps_\xi(x)=(-1)^{\xi(x)},
  \qquad
  W=(\eps_\xi(x))_{\xi\in\whK,\,x\in K},
  \qquad
  L_\xi(z)=\sum_{x\in K}\eps_\xi(x)z_x.
\]
Thus $WW^{\mathsf T}=NI_N$, where $I_N$ is the $N\times N$ identity matrix.
For every $\rho\in\whK$, the Walsh column
$(\eps_\rho(y))_{y\in K}$ is an eigenvector of the group
circulant, because
\[
 C(z)(\eps_\rho(y))_{y\in K}
 =L_\rho(z)(\eps_\rho(x))_{x\in K}.
\]
Equivalently,
\begin{equation}\label{eq:walsh-diagonalization}
 C(z)=N^{-1}W^{\mathsf T}
 \operatorname{diag}(L_\rho(z))_{\rho\in\whK}W.
\end{equation}
For the full circulant this turns the determinant into the product of its
Fourier eigenvalues.  Deleting rows and columns breaks complete translation
symmetry, but the complementary-minor formula below retains a controlled
sum over the missing frequencies.  This is the precise link between Fourier
analysis and the punctured determinant.

For a subspace $U\leq\whK$, write
\[
  U^\perp=\{x\in K:\langle\xi,x\rangle=0
  \text{ for every }\xi\in U\}.
\]
For a vector $x\in K$, we also use the dual annihilator
\[
  x^\perp=\{\xi\in\whK:\langle\xi,x\rangle=0\}.
\]
If $X\subseteq\whK$ and $\xi\in\whK$, set
\[
  X+\xi=\{\rho+\xi:\rho\in X\}.
\]
The set of $\xi$ for which $X+\xi=X$ is called the
\emph{translation stabilizer} of $X$.  We write
$\Span\{u,v\}$ for linear span, whereas a mixed bracket
$\langle\xi,x\rangle$ denotes character evaluation.  The one exterior-space
pairing used below is identified explicitly when it appears.

For $Q\subseteq\whK$, put
\[
  b_Q(\mathbf m)
  =[z^{\mathbf m}]
    \prod_{\rho\in\whK\setminus Q}L_\rho(z).
\]
Thus $Q$ is the set of Fourier factors omitted from the full product, and
$b_Q(\mathbf m)$ records the crossing profile after those factors have been
deleted.  The notation $b_Q$ always remembers the whole set $Q$; only in the
two-hole specialization below can its translation class be encoded by one
relative character.
When $|\mathbf m|=N-|Q|$, choose any labeled list
$d_1,\ldots,d_{N-|Q|}$ with profile $\mathbf m$ and put
\[
  W_{\mathbf m}^{Q}
  =(\eps_\rho(d_j))_{\rho\in\whK\setminus Q,\,1\leq j\leq N-|Q|}.
\]
The coefficient--permanent correspondence proved in
\cite[Section~3.2]{TwoHoleManuscript} has the same counting proof after the
arbitrary row set $Q$ is deleted and gives
\begin{equation}\label{eq:coefficient-permanent}
  \mathbf m!\,b_Q(\mathbf m)=\per W_{\mathbf m}^{Q}.
\end{equation}
Indeed, expansion of the product assigns a value $x$ to every remaining
row; labeling the $m_x$ equal occurrences yields exactly
$\prod_{x\in K}m_x!=\mathbf m!$ corresponding permanent terms.

Translation of every deleted character by the same amount will be used
throughout the remaining structural sections.  Its only effect is the predictable phase
determined by the profile sum.

For two deleted characters, this covariance was proved in
\cite[Section~4, ``Character shift'']{TwoHoleManuscript}.  The same argument
gives the form below for an arbitrary deleted set $Q$.

\begin{lemma}[Character translation]\label{lem:character-translation}
Let $Q\subseteq\whK$ and let $\xi\in\whK$.  For every profile
$\mathbf m$ of size $N-|Q|$,
\[
 b_{Q+\xi}(\mathbf m)
 =\eps_\xi\bigl(\sigma(\mathbf m)\bigr)b_Q(\mathbf m).
\]
\end{lemma}

\begin{proof}
Translating every character index in the product defining $b_Q$ by $\xi$
replaces each variable $z_x$ by $\eps_\xi(x)z_x$.  Taking the coefficient of
$z^{\mathbf m}$ therefore contributes the sign
\[
 \prod_{x\in K}\eps_\xi(x)^{m_x}
 =\eps_\xi\bigl(\sigma(\mathbf m)\bigr).
\]
\end{proof}

Let $\mathcal E_{\mathbb R}$ be the real vector space with basis
$\{e_\rho:\rho\in\whK\}$, and work in its real exterior algebra
$\bigwedge_{\mathbb R}\mathcal E_{\mathbb R}$.  Thus $e_\rho$ is the
standard unit vector in the coordinate indexed by the character $\rho$; for
example, $e_{01}$ is the basis vector labelled by $01$, not the character
$01$ itself.  The wedge product is bilinear
and alternating: $u\wedge u=0$ and $u\wedge v=-v\wedge u$ for degree-one
vectors.  More generally, homogeneous elements of degrees $i$ and $j$
commute with sign $(-1)^{ij}$; in particular, two-forms commute.

No further background on exterior algebras is needed here.  One may read
$u_1\wedge\cdots\wedge u_h$ as an alternating record of $h$ vectors: a
repeated vector makes it zero, and the coefficient of a basis wedge is the
corresponding determinant of coordinates.  These are precisely the two
features used below---they remove endpoint collisions automatically and
produce the corresponding Walsh minors.

Fix once and for all an arbitrary total order $\prec$ on $\whK$.  This order
serves only to orient determinants and exterior basis vectors: no character
is intrinsically smaller than another.  Write $\binom{\whK}{h}$ for the
family of $h$-element subsets of $\whK$.  If
$Q=\{\rho_1\prec\cdots\prec\rho_h\}\in\binom{\whK}{h}$, define
\[
  e_Q=e_{\rho_1}\wedge\cdots\wedge e_{\rho_h}.
\]
For $F=\sum_{Q\in\binom{\whK}{h}}f_Qe_Q$, its support is
$\operatorname{supp}(F)=\{Q\in\binom{\whK}{h}:f_Q\neq0\}$.

For $a\in K$, let
\[
 w_a=\sum_{\rho\in\whK}\eps_\rho(a)e_\rho\in\mathcal E_{\mathbb R}
\]
be the Walsh column indexed by $a$, regarded as an exterior-algebra vector.
For ordered $h$-tuples
$\mathbf a=(a_1,\ldots,a_h)$ and
$\mathbf b=(b_1,\ldots,b_h)$ of distinct points, write
\[
 A(\mathbf a)=\{a_1,\ldots,a_h\},
 \qquad
 B(\mathbf b)=\{b_1,\ldots,b_h\}.
\]
For
$Q=\{\rho_1\prec\cdots\prec\rho_h\}\in\binom{\whK}{h}$, set
\[
  \Delta_Q(\mathbf a)
  =\det(\eps_{\rho_i}(a_j))_{1\leq i,j\leq h}.
\]
Equivalently, $\Delta_Q(\mathbf a)$ is the coefficient of $e_Q$ in
$w_{a_1}\wedge\cdots\wedge w_{a_h}$.
These symbols occur only in the following transfer identity and in the edge
filters that use it.

The next identity is the arbitrary-$h$ complementary-minor expansion
established in~\cite[Section~4, ``Complementary character-minor
expansion'']{TwoHoleManuscript}; it is restated here in the notation needed
for the four- and eight-hole filters.

\begin{lemma}[General complementary-minor expansion]
\label{lem:general-compound}
For every pair of ordered $h$-tuples $\mathbf a,\mathbf b$ of distinct
elements of $K$, there is a sign $\tau(\mathbf a,\mathbf b)\in\{\pm1\}$
such that
\begin{equation}\label{eq:general-compound}
 \tau(\mathbf a,\mathbf b)
 \det C(z)_{A(\mathbf a)^c,B(\mathbf b)^c}
 =\frac1{N^h}\sum_{Q\in\binom{\whK}{h}}
   \Delta_Q(\mathbf a)\Delta_Q(\mathbf b)
   \prod_{\rho\in\whK\setminus Q}L_\rho(z).
\end{equation}
\end{lemma}

\begin{proof}
Apply Cauchy--Binet to the complementary minor in
\eqref{eq:walsh-diagonalization}.  Jacobi's identity replaces both
large Walsh minors by the complementary $h\times h$ determinants on $Q$.
Their signs have a product independent of $Q$, and
$\det(W)^2=N^N$ leaves the factor $N^{-h}$.
\end{proof}

After extracting the coefficient of $z^{\mathbf m}$, a summand has three
separate meanings:
\[
 \underbrace{\Delta_Q(\mathbf a)}_{\text{left boundary}}
 \quad
 \underbrace{\Delta_Q(\mathbf b)}_{\text{right boundary}}
 \quad
 \underbrace{b_Q(\mathbf m)}_{\text{crossing profile}}.
\]
Fourier analysis has therefore separated the two boundary placements from
the prescribed crossing data.  The remaining problem is to sum these terms
without averaging over illegal endpoint configurations.

The determinant identity still sums over arbitrary deleted endpoint sets.
We next force those endpoints to form edges with the prescribed internal
directions.  For $p\in K$ and
$\eta\in\whK$, define the Fourier edge filter
\[
    \Omega_{\eta,p}
    =\sum_{a\in K}\eps_\eta(a)
      w_a\wedge w_{a+p}.
\]
The wedge product makes every term with a repeated endpoint vanish.

The exterior-algebra formulation below packages the two-by-two Walsh-minor
formula from~\cite[Section~4, ``A two-by-two Walsh
minor'']{TwoHoleManuscript} for repeated use at higher deletion levels.

\begin{lemma}[Fourier edge filter]\label{lem:edge-filter}
Let $\eta\in\whK$ and $p\in K$.
\begin{enumerate}[label=\textup{(\roman*)}]
    \item If $\langle\eta,p\rangle=0$, then $\Omega_{\eta,p}=0$.
    \item If $\langle\eta,p\rangle=1$, then
    $\operatorname{supp}(\Omega_{\eta,p})$ consists exactly of the character
    pairs $\{\xi,\xi+\eta\}$.  For such a pair, its
    coefficient has absolute value $2N$ and, up to orientation, equals
    $-2N\eps_\xi(p)$.
\end{enumerate}
\end{lemma}

\begin{proof}
For distinct $\xi,\psi$, the coefficient of
$e_\xi\wedge e_\psi$ in $w_a\wedge w_{a+p}$ is
\[
\eps_{\xi+\psi}(a)
\bigl(\eps_\psi(p)-\eps_\xi(p)\bigr).
\]
After multiplication by $\eps_\eta(a)$ and summation over $a$, character
orthogonality forces $\xi+\psi=\eta$.  Substituting
$\psi=\xi+\eta$ gives
$N\eps_\xi(p)(\eps_\eta(p)-1)$, proving both assertions.
\end{proof}

We call $\Omega_{\eta,p}$ \emph{legal} when
$\langle\eta,p\rangle=1$, and a wedge of edge filters legal when every
factor is legal.

Translation of the character labels acts diagonally on an edge filter.  This
simple covariance will later make the filter phase cancel the translation
phase of $b_Q(\mathbf m)$.

\begin{lemma}[Shift covariance]\label{lem:shift-covariance}
For $\zeta\in\whK$, let $S_\zeta(e_\rho)=e_{\rho+\zeta}$ and extend
$S_\zeta$ to the exterior algebra.  If $\langle\eta,p\rangle=1$, then
\[
 S_\zeta\Omega_{\eta,p}=\eps_\zeta(p)\Omega_{\eta,p}.
\]
More generally, if every factor below is legal, then
\[
 F=\bigwedge_{i=1}^j\Omega_{\eta_i,p_i}
 \quad\Longrightarrow\quad
 S_\zeta F
 =\eps_\zeta(p_1+\cdots+p_j)F.
\]
\end{lemma}

\begin{proof}
Lemma~\ref{lem:edge-filter} gives the representative-free identity
\[
 \Omega_{\eta,p}
 =-N\sum_{\xi\in\whK}
   \eps_\xi(p)e_\xi\wedge e_{\xi+\eta}.
\]
Apply $S_\zeta$ and reindex by $\rho=\xi+\zeta$ to obtain
\[
 S_\zeta\Omega_{\eta,p}
 =-N\sum_{\rho\in\whK}
   \eps_{\rho+\zeta}(p)e_\rho\wedge e_{\rho+\eta}
 =\eps_\zeta(p)\Omega_{\eta,p}.
\]
The general statement follows by applying this identity to every wedge
factor and multiplying the resulting Walsh signs.
\end{proof}

An edge filter is a search device, not a claim that every encoded placement
works.  Wedge products discard placements with repeated endpoints, and the
remaining terms form a weighted sum over all locally legal placements.  For
two degree-$h$ filters
\[
 F=\sum_{Q\in\binom{\whK}{h}}f_Qe_Q,
 \qquad
 G=\sum_{Q\in\binom{\whK}{h}}g_Qe_Q,
\]
define the filtered boundary scalar
\begin{equation}\label{eq:Phi-definition}
 \Phi_{\mathbf m}(F,G)
 =\sum_{Q\in\binom{\whK}{h}}f_Qg_Qb_Q(\mathbf m).
\end{equation}
This is the central quantity in every later case section.  A nonzero value cannot
identify the successful placement by itself.  Its intended role is to certify
that at least one locally legal placement has survived; the next section
proves that implication.

\subsection{The filtered-minor reduction}

The preceding section supplied the algebraic dictionary: complementary
minors of $C(z)$ encode residual crossing bijections, Fourier
diagonalization separates the two deleted boundaries from the crossing
profile, and edge filters range only over locally legal internal-edge
placements.  It has not yet shown that the filtered scalar
$\Phi_{\mathbf m}(F,G)$ from \eqref{eq:Phi-definition} really detects a
successful boundary.

This section proves that common detection bridge.  We first
derive the exact bridge from $\Phi_{\mathbf m}(F,G)$ back to the prescribed
coefficient of a punctured matrix $C(z)_{A^c,B^c}$.  We then place the
two-hole proof on the same map, making clear what changes when four or eight
characters are deleted.  Finally, we establish the coefficient
nonvanishing laws that will be reused in the later case sections; those sections
will construct the filters that make $\Phi_{\mathbf m}(F,G)$ nonzero.

\label{subsec:proof-map}

We begin by separating what is already known from what remains to prove.
The preceding subsections established three ingredients.  First, for a fixed
locally legal boundary $(A,B)$, it is enough to prove the certificate
\[
 [z^{\mathbf m}]\det C(z)_{A^c,B^c}\neq0.
\]
Indeed, it gives the residual crossing bijection, and
Lemma~\ref{lem:boundary-assembly} then gives a solution.  Second,
Lemma~\ref{lem:general-compound} expresses each such punctured determinant
in the Walsh basis.  Third, the edge filters and the scalar
$\Phi_{\mathbf m}(F,G)$ have been defined.  The new claim of this subsection
is that, when the filters encode locally legal placements,
\[
 \Phi_{\mathbf m}(F,G)\neq0
 \quad\Longrightarrow\quad
 [z^{\mathbf m}]\det C(z)_{A^c,B^c}\neq0
\]
for at least one of those placements.  Proposition~\ref{prop:filtered-minor}
will state and prove this implication formally.

Recall how the already established ingredients fit together.  Once the
internal edges occupy endpoint sets $A$ on the left and $B$ on the right,
deleting those rows and columns from $C(z)$ leaves the matrix
\[
 C(z)_{A^c,B^c}=(z_{x+y})_{x\in K\setminus A,\ y\in K\setminus B}.
\]
Its rows and columns are precisely the still-unmatched vertices, and its
entry $z_{x+y}$ records the reduced difference of the possible crossing edge
from $x$ to $y$.  Hence
$[z^{\mathbf m}]\det C(z)_{A^c,B^c}$ is the signed sum over bijections
$\pi:K\setminus A\to K\setminus B$ whose multiset of differences
$x+\pi(x)$ is the prescribed reduced profile $\mathbf m$.  Therefore
$[z^{\mathbf m}]\det C(z)_{A^c,B^c}\neq0$ proves that one such bijection
exists, and
Lemma~\ref{lem:boundary-assembly} then combines it with the internal edges to
give a solution.  This fixed-boundary implication was already established in
Section~\ref{sec:framework}.

What remains is to find a locally legal pair $A,B$ for which this coefficient
is nonzero.  Let $\mathbf a=(a_1,\ldots,a_h)$ list the endpoints used by the
left internal edges, in their paired order, and let $\mathbf b$ do the same on
the right.  Put $A=A(\mathbf a)$ and $B=B(\mathbf b)$.
The complementary-minor identity \eqref{eq:general-compound} was proved in
Section~\ref{sec:framework}; extracting its coefficient of $z^{\mathbf m}$ gives
\[
 \tau(\mathbf a,\mathbf b)
 [z^{\mathbf m}]\det C(z)_{A^c,B^c}
 =\frac1{N^h}
   \sum_{Q\in\binom{\whK}{h}}
   \Delta_Q(\mathbf a)\Delta_Q(\mathbf b)b_Q(\mathbf m).
\]
Thus this is the already-established Fourier expansion for one fixed
boundary; it does not yet choose a boundary for which the left-hand side is
nonzero.

We now make the new move.  Instead of testing $\mathbf a$ and $\mathbf b$ one
pair at a time, the edge filters form weighted sums over all locally legal
endpoint lists.  For scalar
weights $\lambda_{\mathbf a},\mu_{\mathbf b}$, let
\[
 F=\sum_{\mathbf a}\lambda_{\mathbf a}
     w_{a_1}\wedge\cdots\wedge w_{a_h}
   =\sum_Q f_Qe_Q
 \quad\text{and}\quad
 G=\sum_{\mathbf b}\mu_{\mathbf b}
     w_{b_1}\wedge\cdots\wedge w_{b_h}
   =\sum_Q g_Qe_Q,
\]
then $f_Q=\sum_{\mathbf a}\lambda_{\mathbf a}\Delta_Q(\mathbf a)$ and
$g_Q=\sum_{\mathbf b}\mu_{\mathbf b}\Delta_Q(\mathbf b)$.  Substituting these
coefficient expansions into the preceding fixed-boundary identity and
summing over $\mathbf a,\mathbf b$ shows that
$N^{-h}\Phi_{\mathbf m}(F,G)$ is exactly a weighted sum of the prescribed
coefficients of the matrices $C(z)_{A^c,B^c}$, but only for locally legal
boundaries.  The proposition below records the full identity and proves the
new implication.  The complete map is
\[
\begin{aligned}
 \Phi_{\mathbf m}(F,G)\neq0
 &\ \Longrightarrow\
 \text{some locally legal }A,B\text{ satisfy }
 [z^{\mathbf m}]\det C(z)_{A^c,B^c}\neq0\\
 &\ \Longrightarrow\
 \text{there is a bijection }\pi:K\setminus A\longrightarrow K\setminus B\\[-2pt]
 &\hspace{42mm}\text{with difference profile }\mathbf m\\
 &\ \Longrightarrow\
 \text{a solution}.
\end{aligned}
\]

\begin{proposition}[General filtered boundary reduction]
\label{prop:filtered-minor}
Let $h$ be even, let $\mathbf m$ be a profile of size $N-h$, and prescribe
directions
\[
 p_1,\ldots,p_{h/2}\in K\setminus\{0\}
 \quad\text{on the left, and}\quad
 q_1,\ldots,q_{h/2}\in K\setminus\{0\}
 \quad\text{on the right}.
\]
Suppose $F$ is a linear combination of endpoint wedges
\[
 w_{a_1}\wedge w_{a_1+p_1}\wedge\cdots\wedge
 w_{a_{h/2}}\wedge w_{a_{h/2}+p_{h/2}},
\]
and $G$ is the analogous linear combination for
$q_1,\ldots,q_{h/2}$.  If
\[
    \Phi_{\mathbf m}(F,G)\neq0,
\]
then there are $h$-element endpoint sets $A,B\subseteq K$, each a union of
the corresponding $h/2$ disjoint prescribed edges, such that
\[
 [z^{\mathbf m}]\det C(z)_{A^c,B^c}\neq0.
\]
Consequently the two boundary matchings and a crossing bijection on the
remaining vertices assemble to a solution.
\end{proposition}

\begin{proof}
Choose endpoint-wedge expansions
\[
 F=\sum_{\mathbf a}\lambda_{\mathbf a}
      w_{a_1}\wedge\cdots\wedge w_{a_h},
 \qquad
 G=\sum_{\mathbf b}\mu_{\mathbf b}
      w_{b_1}\wedge\cdots\wedge w_{b_h},
\]
where the ordered tuples have distinct entries and encode the prescribed
internal edges.  The coefficient of $e_Q$ in the wedge indexed by
$\mathbf a$ is $\Delta_Q(\mathbf a)$.  Taking the coefficient of
$z^{\mathbf m}$ in Lemma~\ref{lem:general-compound} therefore gives the exact
bridge
\[
 \frac1{N^h}\Phi_{\mathbf m}(F,G)
 =\sum_{\mathbf a,\mathbf b}
   \lambda_{\mathbf a}\mu_{\mathbf b}\tau(\mathbf a,\mathbf b)
   [z^{\mathbf m}]\det C(z)_{A(\mathbf a)^c,B(\mathbf b)^c}.
\]
If the left-hand side is nonzero, at least one complementary-minor
coefficient on the right is nonzero, and its two endpoint sets are locally
legal by construction.  The determinant certificate gives the residual
bijection, and Lemma~\ref{lem:boundary-assembly} completes the matching.
\end{proof}

\subsection{Two holes on the same map}
\label{subsec:two-holes-on-map}

Before turning to four deleted vertices, it is useful to see how the
two-hole argument fits the preceding framework.  This also isolates exactly
what ceases to be automatic in the higher cases.  The notation of
\cite{TwoHoleManuscript} used $H$ for the reduced binary group and
$n=|H|$ for its order; these are the present $K$ and $N=|K|$.

Let $p,q\in K\setminus\{0\}$ be the two internal directions.  For
basepoints $a,b\in K$, the deleted endpoint sets are
\[
 A_a=\{a,a+p\},\qquad B_b=\{b,b+q\},
\]
and admissibility gives
\[
 |\mathbf m|=N-2,
 \qquad
 \sigma(\mathbf m)=p+q.
\]
The character side is particularly simple.  Every
$Q\in\binom{\whK}{2}$ has a unique nonzero difference $\eta$ and may be
written
\[
 Q=\{\xi,\xi+\eta\}.
\]
Following the earlier notation, put
\[
 b_\eta(\mathbf m):=b_{\{0,\eta\}}(\mathbf m).
\]
Thus $b_\eta$ is not a coefficient with one character deleted.  It is the
coefficient attached to the normalized two-character set $\{0,\eta\}$.
The character-translation identity gives
\begin{equation}\label{eq:two-hole-character-shift}
 b_{\{\xi,\xi+\eta\}}(\mathbf m)
 =(-1)^{\langle\xi,\sigma(\mathbf m)\rangle}b_\eta(\mathbf m).
\end{equation}

The two small Walsh minors impose the boundary conditions.  The minor on
$\{a,a+p\}$ vanishes unless $\langle\eta,p\rangle=1$, and the minor on
$\{b,b+q\}$ vanishes unless $\langle\eta,q\rangle=1$.  When both survive,
their product contains the phase
\[
 (-1)^{\langle\xi,p+q\rangle
          +\langle\eta,a+b\rangle}.
\]
The $\xi$-phase cancels the phase in
\eqref{eq:two-hole-character-shift}, because
$\sigma(\mathbf m)=p+q$.  Grouping the $N/2$ pairs with the same
difference $\eta$ therefore gives, up to the irrelevant ordering sign
$\tau_{a,b}\in\{\pm1\}$,
\begin{equation}\label{eq:two-hole-map-transform}
 \tau_{a,b}[z^{\mathbf m}]\det C(z)_{A_a^c,B_b^c}
 =\frac{2}{N}
 \sum_{\substack{\eta\in\whK\\
       \langle\eta,p\rangle=1\\
       \langle\eta,q\rangle=1}}
 (-1)^{\langle\eta,a+b\rangle}b_\eta(\mathbf m).
\end{equation}

This is an ordinary Walsh transform.  The frequency variable is $\eta$,
while
\[
 t=a+b
\]
is the relative placement of the two internal edges---the analogue of the
position or time variable in the usual Fourier picture, not a physical time.
The two-deleted Walsh coefficient lemma of
\cite[Section~4]{TwoHoleManuscript} shows that the frequency-side function in
\eqref{eq:two-hole-map-transform} is not identically zero.  Since the Walsh
transform is invertible, it is nonzero at some $t\in K$.  Choosing $a,b$
with $a+b=t$ gives the desired nonzero determinant coefficient.

This collapse is special to $|Q|=2$.  Up to translation, a two-element set
is classified by one difference $\eta$, and the translate contributes only
a phase that cancels.  A four- or eight-element set has genuine internal
affine geometry that cannot be encoded by one character.  Fourier
invertibility still supplies the common language, but no longer prevents
cancellation by itself.  The role of the higher edge filters is to isolate
families of $Q$ whose total contribution can be controlled.

\section{Deleted Walsh coefficient laws}
\label{sec:deleted-walsh-prep}

In the proof map, Proposition~\ref{prop:filtered-minor} formalizes the
implication
\[
 \Phi_{\mathbf m}(F,G)\neq0
 \quad\Longrightarrow\quad
 [z^{\mathbf m}]\det C(z)_{A^c,B^c}\neq0
\]
for some locally legal boundary; the later case sections must still establish its
hypothesis.  Once the filters restrict
$\Phi_{\mathbf m}(F,G)=\sum_Qf_Qg_Qb_Q(\mathbf m)$ to a controlled family of
deleted character sets, that task depends on knowing which coefficients
$b_Q(\mathbf m)$ survive and at what $2$-adic order.  The laws below supply
exactly this arithmetic input for the affine-support and
$\sigma(\mathbf m)\neq0$ cases used at four and eight holes.

There are two targets.  When the deleted characters form an affine
power-of-two set, we will prove nonvanishing exactly under the natural
annihilator condition.  For an arbitrary power-of-two set, we will prove
nonvanishing whenever its character sum pairs nontrivially with
$\sigma(\mathbf m)$.  The sign-permanent lemma below supplies the exact
$2$-adic scale; the deleted-row parity lemma makes every later
minimum-order minor sum even.

For a fixed deleted character set $Q$, the coefficient $b_Q(\mathbf m)$ is
encoded by the permanent of an $(N-|Q|)\times(N-|Q|)$ deleted Walsh matrix.
Ordinary parity is too coarse because these permanents contain a large forced
power of two.  We begin with two auxiliary tools: an exact valuation for
sign permanents of power-of-two-minus-one order, and a parity lemma tailored
to deleted Walsh matrices.  They isolate the first term that can remain
nonzero; exterior algebra appears only inside the second proof.

For a nonzero integer $a$, let $\vTwo(a)$ be its $2$-adic valuation, put
$\vTwo(0)=\infty$, and let $s_2(a)$ denote the number of ones in the binary
expansion of a nonnegative integer $a$.

The following sign-permanent estimate is classical; see Simion--Schmidt
\cite{SimionSchmidt1983} and the divisibility results of Kr\"auter--Seifter
\cite{KrauterSeifter1983}.  It was also used in the two-hole analysis
\cite[Section~4, ``Sign permanents of order $2^q-1$'']{TwoHoleManuscript}.
We reproduce its short congruence proof to keep the later valuation arguments
self-contained.

\begin{lemma}[Sign permanents of order $2^\ell-1$]\label{lem:sign-permanent}
Let $\ell\geq1$, let $n=2^\ell-1$, and let $A\in\{\pm1\}^{n\times n}$.  Then
\[
    \vTwo(\per A)=\vTwo(n!)=n-\ell.
\]
In particular, $\per A\neq0$.
\end{lemma}

\begin{proof}
Let $J_n$ be the $n\times n$ all-ones matrix and write $A=J_n-2B$, with
$B$ a $0$--$1$ matrix.  Expanding according to the positions in which the
second term is used gives
\[
\per(J_n-2B)
=
\sum_{k=0}^{n}(-2)^k(n-k)!
\sum_{\substack{I,J'\subseteq\{1,\ldots,n\}\\|I|=|J'|=k}}
\per B[I,J'].
\]
For $k\geq1$, Legendre's formula gives
$\vTwo(2^k(n-k)!)=n-s_2(n-k)\geq n-\ell+1$, whereas the $k=0$ term satisfies
$\vTwo(n!)=n-s_2(n)=n-\ell$.  Hence
$\per A\equiv n!\pmod{2^{n-\ell+1}}$, proving the claim.
\end{proof}

Lemma~\ref{lem:sign-permanent} is self-contained and applies to every sign
matrix of order $2^\ell-1$.  For related $2$-adic questions concerning
Sylvester--Hadamard permanents, see Chabaud~\cite{Chabaud2018}.

The nonaffine deletion law below needs only one parity fact about the binary
matrix underlying a deleted Walsh matrix.  We record that fact directly,
rather than introduce separate notation for exterior coefficients and total
minor sums.  Exterior algebra over $\F_2$ appears only inside the proof as an
auxiliary encoding: signs disappear, while $a\wedge a=0$ still holds.

\begin{lemma}[Parity cancellation beyond the deleted set]
\label{lem:deleted-row-minor-parity}
Let $E$ be an $r$-dimensional vector space over $\F_2$, with $r\geq3$, let
$Q\subseteq E$, let $n\geq0$, and let $x_1,\ldots,x_n\in E^*$.  Form the
binary matrix
\[
 B=\bigl(x_j(\rho)\bigr)_{
       \substack{\rho\in E\setminus Q\\1\leq j\leq n}}.
\]
For every $k>|Q|$,
\[
 \sum_{\substack{I\subseteq E\setminus Q,\ J\subseteq\{1,\ldots,n\}\\
                  |I|=|J|=k}}
       \per B[I,J]
 \equiv0\pmod2.
\]
\end{lemma}

\begin{proof}
Work in $\bigwedge_{\F_2}E$ and let $t$ be a central indeterminate.  First,
\begin{equation}\label{eq:full-exterior-product}
 \prod_{\rho\in E}(1+t\rho)=1.
\end{equation}
Indeed, choose independent $u,v,w\in E$ and partition $E$ into cosets of
$L=\Span\{u,v,w\}$.  For one such coset, put
$Y=a+\Span\{v,w\}$ and pair $y\in Y$ with $y+u$.  Then
\[
 \prod_{\rho\in a+L}(1+t\rho)
 =(1+tu)^4\prod_{y\in Y}(1+t^2y\wedge u)
 =1+t^2\left(\sum_{y\in Y}y\right)\wedge u
 =1.
\]
Multiplying over the cosets proves \eqref{eq:full-exterior-product}.  Since
$(1+tq)^2=1$ for every $q\in E$, it follows that
\begin{equation}\label{eq:deleted-exterior-product}
 \prod_{\rho\in E\setminus Q}(1+t\rho)
 =\prod_{q\in Q}(1+tq).
\end{equation}
The coefficient of $t^k$ on the right is zero when $k>|Q|$.

Choose dual bases of $E$ and $E^*$ and write $B=YX$, where the rows of $Y$
are the coordinate vectors of $\rho\in E\setminus Q$ and the columns of $X$
are the coordinate vectors of $x_1,\ldots,x_n$.  For $k\geq0$, use the
canonical bilinear pairing
\[
 \langle\ ,\ \rangle_k:\bigwedge^k E\times\bigwedge^k E^*\longrightarrow\F_2,
 \qquad
 \left\langle \rho_1\wedge\cdots\wedge\rho_k,
 x_1\wedge\cdots\wedge x_k\right\rangle_k
 =\det\bigl(x_j(\rho_i)\bigr)_{i,j=1}^k,
\]
extended bilinearly.  Modulo two, permanent and determinant coincide.
Cauchy--Binet therefore identifies the minor sum in the statement with the
pairing of the degree-$k$ coefficients of
\[
 \prod_{\rho\in E\setminus Q}(1+t\rho)
 \quad\text{and}\quad
 \prod_{j=1}^{n}(1+tx_j).
\]
The first coefficient is zero for $k>|Q|$ by
~\eqref{eq:deleted-exterior-product}; hence the pairing and the displayed
minor sum vanish.
\end{proof}

The remaining common input is arithmetic.  The affine coefficients are most
naturally treated once at every power of two.  We first record the two facts
that make their block decomposition exact.

The full-matrix valuation in the next lemma is a binary specialization of the
group-determinant coefficient formula of Wang and Zhang
\cite[Theorem~2.1 and the proof of Theorem~1.3]{WangZhang2026}, which treats
finite abelian groups of prime-power order.  In the binary setting, the same
statement and Walsh-pairing proof appeared in
\cite[Section~3, ``Full character-matrix permanent
valuation'']{TwoHoleManuscript}.  We recall the proof here because its block
recursion is the form used to organize the deleted-coefficient arguments
below.

\begin{lemma}[Full Walsh permanent]\label{lem:full-walsh-permanent}
Let $E\cong\F_2^\ell$, write
$E^*=\operatorname{Hom}_{\F_2}(E,\F_2)$, let $h=|E|=2^\ell$, and let
$y_1,\ldots,y_h\in E^*$.  Define
\[
 P_E(y_1,\ldots,y_h)
 =\per\bigl((-1)^{\langle y_j,\tau\rangle}\bigr)_{
          \substack{\tau\in E\\1\leq j\leq h}},
\]
where $\langle y,\tau\rangle=y(\tau)$ is the evaluation pairing.  If
$y_1+\cdots+y_h\neq0$, then $P_E(y_1,\ldots,y_h)=0$.  If
$y_1+\cdots+y_h=0$, then
\[
 \vTwo\bigl(P_E(y_1,\ldots,y_h)\bigr)=h-1.
\]
\end{lemma}

\begin{proof}
If $y_1+\cdots+y_h\neq0$, translate the row index by an element of $E$ that
pairs nontrivially with this sum.  This permutes the rows and multiplies the
permanent by $-1$, so the permanent is zero.

Assume now that $y_1+\cdots+y_h=0$.  We prove the valuation by induction on
$\ell$.  For $\ell=0$, the matrix is $[1]$.  Suppose $\ell\geq1$, choose
$0\neq\eta\in E$, and put $\overline E=E/\Span\{\eta\}$.  For
$a\in\F_2$, let
\[
 n_a=\bigl|\{j:y_j(\eta)=a\}\bigr|.
\]
Let $\Pi_\eta$ be the set of pair partitions of the column labels in which
paired columns have the same value on $\eta$.  For
$\mathcal P\in\Pi_\eta$ and $P=\{i,j\}\in\mathcal P$, the functional
$\bar y_P=y_i+y_j$ annihilates $\eta$ and hence belongs to
$\overline E^*$.  Pairing the rows as $\{\tau,\tau+\eta\}$ gives
\begin{equation}\label{eq:full-walsh-recursion}
 P_E(y_1,\ldots,y_h)
 =2^{h/2}(-1)^{n_1/2}
  \sum_{\mathcal P\in\Pi_\eta}
  P_{\overline E}\bigl((\bar y_P)_{P\in\mathcal P}\bigr).
\end{equation}
Indeed, for one row pair and two assigned columns, the sum over the two
assignments is zero when their values on $\eta$ differ, while in the compatible
case it is
\[
 2(-1)^{y_i(\eta)}(-1)^{(y_i+y_j)(\tau)}.
\]
Multiplying these factors and summing over the assignments of column pairs to
row pairs yields \eqref{eq:full-walsh-recursion}.

The zero-sum hypothesis makes $n_1$ even, and $h$ even makes $n_0$ even.
Thus
\[
 |\Pi_\eta|=(n_0-1)!!(n_1-1)!!
\]
is odd, with $(-1)!!=1$.  Every compressed list
$(\bar y_P)_{P\in\mathcal P}$ has zero sum.  By induction, each permanent in
the sum in \eqref{eq:full-walsh-recursion} is $2^{h/2-1}$ times an odd integer.
The sum contains an odd number of odd normalized terms, so it remains odd.
Therefore
\[
 \vTwo(P_E)=\frac h2+\left(\frac h2-1\right)=h-1.
\]
\end{proof}

After transposition, for $h>2$ the nonvanishing criterion is also the
suppression law for repeated rows of a Sylvester matrix proved by
Crespi~\cite{Crespi2015}.  The induction above gives the exact valuation
directly in Walsh coordinates.  None of these full-matrix results determines
the affine or nonaffine deleted coefficients developed below or settles the
boundary-selection problem for the unrestricted four- and eight-hole
theorems.  For
related $2$-adic questions concerning Sylvester--Hadamard permanents, see
Chabaud~\cite{Chabaud2018}.

\begin{lemma}[Odd number of zero-sum block partitions]
\label{lem:block-partition-parity}
Let $E=\F_2^\ell$, put $h=2^\ell$, let $k\geq1$, and let
$v_1,\ldots,v_{kh}\in E$ be labeled elements with zero total sum.  The
number of partitions of the labels into $k$ unlabeled blocks of size $h$,
each having zero sum in $E$, is odd.
\end{lemma}

\begin{proof}
We first show that, for the fixed labeled element $v_1$, the number of
zero-sum $h$-subsets containing $v_1$ is odd.  Write the remaining elements
as $w_1,\ldots,w_{kh-1}$, so $\sum_jw_j=v_1$.

Work in the group algebra $\mathcal G_E=\F_2[E]$.  For $v\in E$, write $X^v$ for the
corresponding basis element and put $n_v=X^v+1$.  If $\mathfrak J$ is the
augmentation ideal, then $n_v\in\mathfrak J$.  Choosing a basis
$e_1,\ldots,e_\ell$ of $E$ and
putting $Y_i=X^{e_i}+1$ gives
\[
 \mathcal G_E\cong
 \F_2[Y_1,\ldots,Y_\ell]/(Y_1^2,\ldots,Y_\ell^2),
 \qquad \mathfrak J^{\ell+1}=0.
\]
Consider
\[
 P(u)=\prod_{j=1}^{kh-1}(1+uX^{w_j})\in \mathcal G_E[u].
\]
Expanding according to the factors from which an $n_{w_j}$ is selected gives
\[
 [u^{h-1}]P(u)
 =\sum_{d=0}^{\ell}
   \binom{kh-1-d}{h-1-d}
   \sum_{\substack{A\subseteq[kh-1]\\|A|=d}}
   \prod_{j\in A}n_{w_j}.
\]
Since $h=2^\ell$, Lucas' theorem gives
\[
 \binom{kh-1-d}{h-1-d}\equiv1\pmod2
 \qquad(0\leq d\leq\ell).
\]
Therefore
\[
 [u^{h-1}]P(u)
 =\prod_{j=1}^{kh-1}(1+n_{w_j})
 =\prod_{j=1}^{kh-1}X^{w_j}
 =X^{v_1}.
\]
Thus an odd number of $(h-1)$-subsets of the $w_j$ have sum $v_1$, proving
the claim.

Now induct on $k$.  The case $k=1$ is immediate.  There are an odd number
of choices for the block containing $v_1$, and every complement again has
zero total sum.  By induction each complement has an odd number of admissible
partitions.  The total number is therefore odd.
\end{proof}

The endpoint cases of the next theorem are inherited specializations.  For
$\ell=0$, one has $U=\{0\}$ and $Q=\{\xi\}$; the annihilator condition is
vacuous, and \eqref{eq:coefficient-permanent} reduces the valuation formula
to Lemma~\ref{lem:sign-permanent} for a sign matrix of order $N-1$.  For
$\ell=1$, writing $U=\Span\{\eta\}$ gives
$Q=\{\xi,\xi+\eta\}$ and the condition
$\langle\eta,\sigma(\mathbf m)\rangle=0$; the conclusion is exactly the
two-deleted Walsh-coefficient law of
~\cite[Section~4]{TwoHoleManuscript}.  The new deletion levels used below are
$\ell=2$ and $\ell=3$; the proof treats every $\ell<r$ uniformly.

\begin{theorem}[Affine power-of-two deletion]
\label{thm:affine-power-two-deletion}
Let $K\cong\F_2^r$, $N=2^r$, and let $U\leq\whK$ have dimension
$\ell<r$ and order $h=2^\ell$.  For $\xi\in\whK$, put $Q=\xi+U$.
For a profile $\mathbf m$
of size $N-h$,
\[
 b_Q(\mathbf m)=0
 \quad\text{if}\quad
 \sigma(\mathbf m)\notin U^\perp.
\]
If $\sigma(\mathbf m)\in U^\perp$, then
\[
 \vTwo\bigl(b_Q(\mathbf m)\bigr)
 =\vTwo\!\left(\frac{(N-h)!}{\mathbf m!}\right),
\]
and in particular $b_Q(\mathbf m)\neq0$.
\end{theorem}

\begin{proof}
Form the $(N-h)\times(N-h)$ repeated-column matrix
$W_{\mathbf m}^{Q}$.  If
$\sigma(\mathbf m)\notin U^\perp$, choose $\eta\in U$ with
$\langle\eta,\sigma(\mathbf m)\rangle=1$.  Translation of the row index
by $\eta$ preserves $\whK\setminus Q$ and multiplies the permanent by
$-1$.  Hence $\per W_{\mathbf m}^{Q}=0$.

Assume $\sigma(\mathbf m)\in U^\perp$.  Translating all character indices
changes the permanent only by an overall column sign, so it suffices to take
$Q=U$.  Put
\[
 M=\frac Nh=2^{r-\ell}.
\]
The remaining rows split into the $M-1$ nontrivial cosets $\gamma+U$.  For
$x\in K$, let
$\bar x\in U^*=\operatorname{Hom}_{\F_2}(U,\F_2)$ be the restriction of
$x$ to $U$.  The
hypothesis is equivalent to
\[
 \sum_xm_x\bar x=0\qquad\text{in }U^*.
\]

Partition the labeled columns into $M-1$ unlabeled blocks of size $h$.
Lemma~\ref{lem:full-walsh-permanent} shows that a row block $\gamma+U$
contributes zero unless the restricted column types in its assigned block
have zero sum.  For a surviving partition $\cP$ and a block $P\in\cP$ of
labeled column positions, put $X_P=\sum_{j\in P}d_j$.  Define the sign
matrix $A_{\cP}$ with rows indexed by the nontrivial cosets
$\gamma+U\in(\whK/U)\setminus\{U\}$, columns indexed by the blocks
$P\in\cP$, and entries
\[
 A_{\cP}
 =\bigl(\eps_\gamma(X_P)\bigr)_
 {\substack{\gamma+U\in(\whK/U)\setminus\{U\}\\P\in\cP}}.
\]
Every surviving block has restricted column sum zero, so
$X_P\in U^\perp$.  Consequently each entry depends only on the coset
$\gamma+U$, not on its representative, and $A_{\cP}$ is a well-defined sign
matrix of order $M-1=2^{r-\ell}-1$.

The same lemma gives
\[
 \per W_{\mathbf m}^{U}[\gamma+U,P]
 =2^{h-1}c_P\eps_\gamma(X_P),
\]
where $c_P$ is odd and independent of $\gamma$.  Summing over bijections
between the column blocks and row cosets yields
\[
 2^{(h-1)(M-1)}
 \left(\prod_{P\in\cP}c_P\right)\per A_{\cP}.
\]
Lemma~\ref{lem:sign-permanent} gives
\[
 \vTwo(\per A_{\cP})=(M-1)-(r-\ell).
\]
Every surviving partition therefore contributes with exact valuation
\[
 (h-1)(M-1)+(M-1)-(r-\ell)=N-h-r+\ell.
\]

Lemma~\ref{lem:block-partition-parity}, applied to the labeled restricted
column types in $U^*\cong\F_2^\ell$, says that the number of surviving
partitions is odd.  After division by $2^{N-h-r+\ell}$, each surviving
contribution is odd, so their sum is odd.  Hence
\[
 \vTwo\bigl(\per W_{\mathbf m}^{U}\bigr)=N-h-r+\ell.
\]
Finally, $s_2(N-h)=r-\ell$, and Legendre's formula gives
\[
 \vTwo((N-h)!)=N-h-r+\ell.
\]
Now use \eqref{eq:coefficient-permanent} with $Q=U$ and translate back from
$U$ to $\xi+U$.
\end{proof}

The second common law detects a deleted character set whose character sum
$\sum_{\rho\in Q}\rho$ pairs nontrivially with the crossing-profile sum.  It
contains the four- and
eight-character estimates as the cases $\ell=2$ and $\ell=3$.

\begin{theorem}[Power-of-two deletion with nontrivial character-sum pairing]
\label{thm:nonzero-sum-power-two-deletion}
Let $K\cong\F_2^r$, $N=2^r$, let $2\leq\ell<r$, and put
$h=2^\ell$.  Let $Q\subseteq\whK$ have $h$ elements and let
$|\mathbf m|=N-h$.  If
\[
 \left\langle\sum_{\rho\in Q}\rho,\sigma(\mathbf m)\right\rangle=1,
\]
then
\begin{equation}\label{eq:nonzero-power-two-valuation}
 \vTwo\bigl(b_Q(\mathbf m)\bigr)
 =\vTwo\!\left(\frac{(N-h)!}{\mathbf m!}\right)-(\ell-1).
\end{equation}
In particular, $b_Q(\mathbf m)\neq0$.
\end{theorem}

\begin{proof}
Put $n=N-h$, choose a labeled list $d_1,\ldots,d_n$ with profile
$\mathbf m$, let $J_n$ be the $n\times n$ all-ones matrix, and write the
deleted Walsh matrix as
\[
 W_{\mathbf m}^{Q}=J_n-2B,
 \qquad
 B_{\rho,j}=\langle\rho,d_j\rangle
 \quad(\rho\in\whK\setminus Q,\ 1\leq j\leq n).
\]
As in the sign-permanent expansion,
\[
 \per(J_n-2B)
 =
 \sum_{k=0}^{n}(-2)^k(n-k)!
 \sum_{\substack{I\subseteq\whK\setminus Q,\ J\subseteq\{1,\ldots,n\}\\
                  |I|=|J|=k}}
       \per B[I,J].
\]
 The scalar part of the $k$th term has valuation
\[
 \vTwo\bigl(2^k(n-k)!\bigr)=n-s_2(n-k).
\]
Put $t=n-k$.  Since $0\leq t\leq n=2^r-2^\ell<2^r$, one has
$s_2(t)\leq r-1$: equality $s_2(t)=r$ would force
$t=2^r-1>n$.  Equality $s_2(t)=r-1$ holds precisely when
\[
 t=2^r-1-2^j
 \quad\text{for some }j\in\{\ell,\ldots,r-1\},
\]
because the missing bit must have index at least $\ell$ in order that
$t\leq n$.  Hence the minimum is $n-r+1$, attained exactly for
\[
 k=2^j-2^\ell+1,
 \qquad j=\ell,\ldots,r-1.
\]
The first such index is $k=1$; every later one is at least $h+1$.

Lemma~\ref{lem:deleted-row-minor-parity}, applied with
$E=\whK$ and $x_j(\rho)=\langle\rho,d_j\rangle$, makes the inner minor sum
even at every later minimum index.  At $k=1$, that sum is
\begin{align*}
 &\sum_{\rho\notin Q}\sum_{j=1}^{n}\langle\rho,d_j\rangle\\
 &=\left\langle\sum_{\rho\notin Q}\rho,
                   \sigma(\mathbf m)\right\rangle\\
 &=\left\langle\sum_{\rho\in Q}\rho,
                   \sigma(\mathbf m)\right\rangle
 \equiv1\pmod2.
\end{align*}
Here $\sum_{\rho\in\whK}\rho=0$ because $r\geq3$, which justifies the
third equality.  Thus the $k=1$ term is the unique odd contribution after
division by
$2^{n-r+1}$, and
\[
 \vTwo\bigl(\per W_{\mathbf m}^{Q}\bigr)=n-r+1.
\]
Since $s_2(n)=r-\ell$,
\[
 \vTwo(n!)=n-r+\ell.
\]
The coefficient--permanent identity \eqref{eq:coefficient-permanent} now gives
\eqref{eq:nonzero-power-two-valuation}.
\end{proof}

The restriction $\ell\geq2$ is essential.  The two-hole case follows the
same Fourier map, but its deleted-coefficient arithmetic is different and is
supplied by the two-deleted lemma of~\cite{TwoHoleManuscript}.  The common
power-of-two laws above begin with four deleted characters.

\begin{lemma}[Affine-coset filter coefficient]
\label{lem:affine-filter-walsh}
Let $K\cong\F_2^r$, $N=|K|$, let $t\geq1$, and put $j=2^t$.  Let
$\eta\in\whK\setminus\{0\}$, let $E\leq\whK$ have dimension $t$ with
$\eta\notin E$, and put $U=\Span\{\eta\}\oplus E$.  Suppose
$p_1,\ldots,p_j\in K$ satisfy $\langle\eta,p_i\rangle=1$, put
$d=p_1+\cdots+p_j$, and write
\[
 F=\bigwedge_{i=1}^j\Omega_{\eta,p_i}
   =\sum_{R\in\binom{\whK}{2j}}f_Re_R.
\]
For each $\xi\in\whK$, put $Q=\xi+U$ and define $y_i\in E^*$ by
$y_i(\tau)=\langle\tau,p_i\rangle$.  There is an orientation sign
$\omega_{Q,E,\eta}\in\{\pm1\}$ such that
\[
 f_Q=\omega_{Q,E,\eta}(2N)^j\eps_\xi(d)
      P_E(y_1,\ldots,y_j).
\]
Consequently,
\[
 f_Q\neq0\quad\Longleftrightarrow\quad d\in U^\perp,
 \qquad
 \vTwo(f_Q)=j(r+1)+j-1
\]
whenever $f_Q\neq0$.
\end{lemma}

\begin{proof}
Orient the $\eta$-orbits of $Q$ as
$(\xi+\tau,\xi+\tau+\eta)$, $\tau\in E$.  By
Lemma~\ref{lem:edge-filter}, assigning the $i$th filter to the orbit indexed
by $\tau$ contributes $-2N\eps_{\xi+\tau}(p_i)$, up to one common
orientation sign.  The degree-two orbit factors commute, so summing over all
bijections from the filters to the orbits gives the displayed permanent
formula.  Moreover,
\[
 \sum_{i=1}^j y_i(\tau)=\langle\tau,d\rangle\qquad(\tau\in E).
\]
Lemma~\ref{lem:full-walsh-permanent} therefore makes the permanent nonzero
exactly when $d$ annihilates $E$.  Legality and the evenness of $j$ give
$\langle\eta,d\rangle=j=0$ in $\F_2$, so this is equivalent to
$d\in U^\perp$.  The valuation follows from $N=2^r$ and
$\vTwo(P_E)=j-1$.
\end{proof}

We now package the affine boundary argument used at both levels.  The
parameter $j$ counts prescribed internal edges in each affine half:
$j=2$ means four holes in total, while $j=4$ means eight.

\begin{proposition}[Uniform affine-filter criterion]
\label{prop:uniform-affine-filter}
Let $K\cong\F_2^r$, let $N=|K|$, and let $j\in\{2,4\}$.  Assume $r\geq3$,
and assume $r\geq4$ when $j=4$.  Let
\[
 P=(p_1,\ldots,p_j),
 \qquad
 R=(q_1,\ldots,q_j)
\]
be ordered tuples in $K\setminus\{0\}$, and let $\mathbf m$ be a profile of
size $N-2j$ satisfying
\[
 \Sigma:=\sigma(\mathbf m)
 =
 \sum_{i=1}^{j}p_i+\sum_{i=1}^{j}q_i.
\]
Suppose there are distinct characters $\eta,\theta\in \Sigma^\perp$ such that
\[
 \langle\eta,p_i\rangle=1,
 \qquad
 \langle\theta,q_i\rangle=1
 \qquad(1\leq i\leq j).
\]
Put
\[
 d=\sum_{i=1}^{j}p_i,
 \qquad
 T=\Span\{\eta,\theta\},
\]
so $\dim T=2$: legality makes $\eta$ and $\theta$ nonzero, and two distinct
nonzero vectors over $\F_2$ are linearly independent.  Define
\[
 F=\bigwedge_{i=1}^{j}\Omega_{\eta,p_i},
 \qquad
 G=\bigwedge_{i=1}^{j}\Omega_{\theta,q_i}.
\]
If $j=4$, assume in addition that
\[
 \dim(\Sigma^\perp\cap d^\perp)\geq3.
\]
Then
\[
 \Phi_{\mathbf m}(F,G)\neq0.
\]
Consequently there are locally legal deleted sets $A,B\subseteq K$, each of
size $2j$, for which
\[
 [z^{\mathbf m}]\det C(z)_{A^c,B^c}\neq0.
\]
\end{proposition}

\begin{proof}
Put
\[
 d_R=\sum_{i=1}^{j}q_i=d+\Sigma.
\]
Since $j$ is even and the filters are legal,
\[
 \langle\eta,d\rangle=0,
 \qquad
 \langle\theta,d\rangle
 =\langle\theta,d_R+\Sigma\rangle=0.
\]
Thus
\[
 T\subseteq \Sigma^\perp\cap d^\perp.
\]

By Lemma~\ref{lem:edge-filter}, a set in the support of $F$ is a union of
$j$ $\eta$-pairs, while a set in the support of $G$ is a union of $j$
$\theta$-pairs.  Every common-support set is therefore invariant under
$T$.  When $j=2$, such a four-set is a coset $Q=\xi+T$.  When $j=4$, such
an eight-set is the union of two distinct $T$-cosets and hence has the form
\[
 Q=\xi+U
\]
for a unique three-dimensional subspace $U$ containing $T$.

Fix one of these subspaces $U$, where $U=T$ for $j=2$, and choose a
complement $E$ of $\Span\{\eta\}$ in $U$.  The $\eta$-orbits in
$Q=\xi+U$ are indexed by $E$, so
Lemma~\ref{lem:affine-filter-walsh} gives
\[
 f_Q\neq0
 \quad\Longleftrightarrow\quad
 d\in U^\perp,
 \qquad
 \vTwo(f_Q)=j(r+1)+j-1
\]
whenever the coefficient is nonzero.  The same argument on the right gives
\[
 g_Q\neq0
 \quad\Longleftrightarrow\quad
 d_R\in U^\perp,
 \qquad
 \vTwo(g_Q)=j(r+1)+j-1.
\]
Because $d_R=d+\Sigma$, both coefficients are nonzero exactly when
\begin{equation}\label{eq:uniform-affine-support}
 T\subseteq U\subseteq \Sigma^\perp\cap d^\perp.
\end{equation}

For every $U$ satisfying \eqref{eq:uniform-affine-support},
Theorem~\ref{thm:affine-power-two-deletion} applies to every coset
$\xi+U$, since $\Sigma\in U^\perp$.  Hence
\[
 b_{\xi+U}(\mathbf m)\neq0,
 \qquad
 \vTwo\bigl(b_{\xi+U}(\mathbf m)\bigr)
 =
 \vTwo\!\left(\frac{(N-2j)!}{\mathbf m!}\right).
\]

It remains to rule out cancellation.  Lemma~\ref{lem:shift-covariance}
shows that translation by $\zeta\in\whK$ multiplies the product of the two
filter coefficients by
\[
 \eps_\zeta(d+d_R)=\eps_\zeta(\Sigma);
\]
the two orientation signs cancel in the product.  By
Lemma~\ref{lem:character-translation}, the same translation multiplies the
deleted coefficient by $\eps_\zeta(\Sigma)$.  Therefore
\[
 f_{Q+\zeta}g_{Q+\zeta}b_{Q+\zeta}(\mathbf m)
 =
 f_Qg_Qb_Q(\mathbf m).
\]

When $j=2$, the only possible direction is $U=T$, and hence
\[
 \Phi_{\mathbf m}(F,G)
 =
 \frac N4 f_Tg_Tb_T(\mathbf m)\neq0.
\]
Suppose $j=4$.  The number of three-dimensional subspaces satisfying
\eqref{eq:uniform-affine-support} is
\[
 2^{\dim(\Sigma^\perp\cap d^\perp)-2}-1,
\]
which is positive and odd.  Taking the linear subspace $U$ as the
representative of its translation orbit gives
\[
 \Phi_{\mathbf m}(F,G)
 =
 \frac N8
 \sum_{\substack{\dim U=3\\
                  T\subseteq U\subseteq \Sigma^\perp\cap d^\perp}}
      f_Ug_Ub_U(\mathbf m).
\]
The contribution from every $U$-orbit has the same exact $2$-adic valuation.
After the common power of two is removed, each contribution is odd; an odd
number of odd integers cannot sum to zero.
Thus $\Phi_{\mathbf m}(F,G)\neq0$ also in this case.
Proposition~\ref{prop:filtered-minor} supplies $A$ and $B$.
\end{proof}

The preceding results provide the common Fourier reduction and the two
coefficient laws.  The remaining support geometry is established in the case
sections.  Proposition~\ref{prop:four-holes-affine} and
Lemma~\ref{lem:nonaffine-filter-support} treat the affine and nonaffine
four-character supports.  Sections~\ref{sec:eight-zero} and
~\ref{sec:eight-nonzero} then determine which eight-character supports can
arise and how their contributions are controlled.  These are forward
references to later arguments, not additional hypotheses of the common
framework.

\section{Four holes}\label{sec:four}

The two-hole theorem chooses one internal edge in each affine half.  The
constructive four-hole result in~\cite{TwoHoleManuscript} went further only
when every request value had even multiplicity.  We now remove that
global symmetry: the four internal directions and the crossing profile may
have arbitrary multiplicities, subject only to admissibility.

The point is not that every legal placement works.  The four directions must
be split as
\[
 p_1,p_2\mid q_1,q_2,
\]
where the vertical bar separates the two affine halves, and the endpoints of
the two edges on each side must be chosen so that the prescribed crossing
profile can be realized on the remaining vertices.  The algebraic target is
the following proposition.

\begin{proposition}[Four-hole boundary noncancellation]
\label{prop:four-hole-boundary}
Let $K\cong\F_2^r$ with $r\geq3$ and $N=|K|$.  Let
$h_1,h_2,h_3,h_4\in K\setminus\{0\}$, and let $\mathbf m$ be a profile of
size $N-4$ satisfying
\[
 \sigma(\mathbf m)=h_1+h_2+h_3+h_4.
\]
Then the holes can be ordered as $p_1,p_2,q_1,q_2$, and there are
four-element sets $A,B\subseteq K$ such that $A$ is the union of two
disjoint edges of differences $p_1,p_2$, $B$ is the union of two disjoint
edges of differences $q_1,q_2$, and
\[
 [z^{\mathbf m}]\det C(z)_{A^c,B^c}\neq0.
\]
\end{proposition}

Put
\[
 \Sigma=\sigma(\mathbf m)=h_1+h_2+h_3+h_4.
\]
The proof is organized by the structural obstruction that will reappear for
eight holes.  If $\Sigma$ is not one of the holes, quotienting by $\Span\{\Sigma\}$
turns the four images into a nonzero zero-sum configuration.  The affine
branch constructs legal filters for which the uniform affine criterion
reduces the filtered scalar to a nonzero affine deleted coefficient.
If $\Sigma$ is a hole, its quotient image is zero, so that construction cannot
even choose a legal character for it.  The nonaffine branch instead uses a
four-cycle to reduce the filtered scalar to one nonaffine deleted coefficient,
which is again nonzero.  The next subsection records the two coefficient laws,
and the following two subsections implement these complementary branches.

\subsection{Four-character geometry and coefficient laws}

A four-element subset $Q\subseteq\whK$ is \emph{affine} if it is a coset of
a two-dimensional subspace.  For four distinct points this is equivalent to
$\sum_{\rho\in Q}\rho=0$: after translating one point to $0$, the fourth
must be the sum of the other two independent points.  Thus a four-set is
either an affine plane or has nonzero character sum.  The two common
power-of-two theorems from Section~\ref{sec:common-map} match these two
geometries and give the needed estimates once the filter hypotheses are
arranged.

\begin{example}[The two four-character geometries]\label{ex:four-geometries}
Identify $K$ and $\whK$ with $\F_2^3$ by the standard dot product, and let
$u,v,w$ be the standard basis.  The set
\[
 Q_{\rm aff}=\{0,u,v,u+v\}
\]
is the plane $\Span\{u,v\}$; its character sum is zero and its translation
stabilizer is $\Span\{u,v\}$.  In contrast,
\[
 Q_{\rm non}=\{0,u,v,w\}
\]
has character sum $u+v+w\neq0$ and trivial translation stabilizer.  These
are the two geometries encountered in the degree-four expansion.
\end{example}

\begin{corollary}[Affine four-deleted Walsh coefficient]
\label{cor:affine-four-deleted}
Assume $r\geq3$.  Let $Q=\xi+T\subseteq\whK$, where $\dim T=2$, and let
$|\mathbf m|=N-4$.  Then
\[
 b_Q(\mathbf m)=0
 \quad\text{if }\sigma(\mathbf m)\notin T^\perp.
\]
If $\sigma(\mathbf m)\in T^\perp$, then
\[
 \vTwo\bigl(b_Q(\mathbf m)\bigr)
 =
 \vTwo\!\left(\frac{(N-4)!}{\mathbf m!}\right),
\]
and in particular $b_Q(\mathbf m)\neq0$.
\end{corollary}

\begin{proof}
Apply Theorem~\ref{thm:affine-power-two-deletion} with $\ell=2$ and $U=T$.
\end{proof}

\begin{corollary}[Nonaffine four-deleted Walsh coefficient]\label{cor:nonaffine-four-deleted}
Assume $r\geq3$.  Let $Q\subseteq\whK$ have four elements and let
$|\mathbf m|=N-4$.  If
\[
    \left\langle\sum_{\rho\in Q}\rho,\sigma(\mathbf m)\right\rangle=1,
\]
then
\[
\vTwo\bigl(b_Q(\mathbf m)\bigr)
=
\vTwo\left(\frac{(N-4)!}{\mathbf m!}\right)-1.
\]
In particular, $b_Q(\mathbf m)\neq0$.
\end{corollary}

\begin{proof}
Apply Theorem~\ref{thm:nonzero-sum-power-two-deletion} with $\ell=2$.
\end{proof}

In Example~\ref{ex:four-geometries}, taking $\sigma(\mathbf m)=w$ makes both deleted
coefficients survive, but Corollaries~\ref{cor:affine-four-deleted} and
\ref{cor:nonaffine-four-deleted} place them one power of $2$ apart.  This is
why the proof cannot average the affine and nonaffine sectors together.

\subsection{The affine branch: the total sum is not a hole}

Suppose that $\Sigma$ is not one of the four hole directions.  In
$K/\Span\{\Sigma\}$ the hole images remain nonzero and have zero total sum.
The next elementary lemma chooses the two legal character directions in that
zero-sum quotient.  Their lifts lie in $\Sigma^\perp$, so the uniform affine
criterion closes the branch.

\begin{lemma}[Character choice when the four holes sum to zero]
\label{lem:zero-sum-character-choice}
Assume $r\geq2$.  Let $h_1,h_2,h_3,h_4\in K\setminus\{0\}$ satisfy
\[
 h_1+h_2+h_3+h_4=0.
\]
Then the four labeled vectors can be ordered as
$p_1,p_2,q_1,q_2$, and there exist distinct characters
$\eta,\theta\in\whK$ such that
\[
 \langle\eta,p_1\rangle=\langle\eta,p_2\rangle=1,
 \qquad
 \langle\theta,q_1\rangle=\langle\theta,q_2\rangle=1.
\]
\end{lemma}

\begin{proof}
Suppose first that $r=2$.  Write the three nonzero vectors of $K$ as
$a,b,a+b$.  The zero-sum condition says that their three multiplicities have
the same parity; since their total is four, all three multiplicities are even.
Thus the multiset has the form $\{a,a,a,a\}$ or $\{a,a,b,b\}$ after relabeling
$a,b$.  Order equal copies together.  For each nonzero $x\in K$, the set
$\{\chi\in\whK:\langle\chi,x\rangle=1\}$ has two elements, so distinct
characters can be chosen for the two pairs.

Now assume $r\geq3$ and take any ordering $p_1,p_2,q_1,q_2$.  The affine
system $\chi(p_1)=\chi(p_2)=1$ is consistent: if $p_1=p_2$ it is one
nontrivial equation, and otherwise $p_1,p_2$ are linearly independent.  Its
solution set has at least two elements.  The same holds for
$\chi(q_1)=\chi(q_2)=1$.  Choose $\eta$ from the first solution set and then
$\theta\neq\eta$ from the second.
\end{proof}

\begin{proposition}[Affine boundary away from the total-sum hole]
\label{prop:four-holes-affine}
Assume $r\geq3$.  Let $h_1,h_2,h_3,h_4\in K\setminus\{0\}$, put
\[
 \Sigma=h_1+h_2+h_3+h_4,
\]
and assume $\Sigma\notin\{h_1,h_2,h_3,h_4\}$.  If $\mathbf m$ is a profile of
size $N-4$ with $\sigma(\mathbf m)=\Sigma$, then, after a suitable ordering
$p_1,p_2,q_1,q_2$ of the holes, there are deleted vertex sets
$A,B\subseteq K$, each of size four, such that $A$ realizes
$p_1,p_2$, $B$ realizes $q_1,q_2$, and
\[
 [z^{\mathbf m}]\det C(z)_{A^c,B^c}\neq0.
\]
\end{proposition}

\begin{proof}
Pass to $\overline K=K/\Span\{\Sigma\}$.  The four images are nonzero and have
zero total sum.  Since $\dim\overline K=r-1\geq2$,
Lemma~\ref{lem:zero-sum-character-choice} gives an ordering
$\bar p_1,\bar p_2,\bar q_1,\bar q_2$ and distinct quotient characters
$\eta,\theta$ taking value one on their corresponding pairs.  Viewed on
$K$, these characters belong to $\Sigma^\perp$.  Proposition
~\ref{prop:uniform-affine-filter}, with $j=2$, now gives the required sets
$A,B$.
\end{proof}

\subsection{The nonaffine branch: the total sum is a hole}

Suppose now that $\Sigma$ occurs among the four holes.  Its image in
$K/\Span\{\Sigma\}$ is zero, so no quotient character can take value one on that
hole and the affine construction genuinely stops.  We instead choose the
split and four filter directions so that the common support is a single
translation orbit of nonaffine four-sets.  Its filter phase cancels the
translation phase of $b_Q(\mathbf m)$, so all orbit terms reinforce one
another.

The remaining hole geometry has two structural forms, distinguished by
whether $\Sigma$ lies in the plane generated by the other three holes.  The next
lemma chooses the filters in both forms.

\begin{lemma}[Character choice when the total sum is a hole]
\label{lem:nonzero-sum-character-choice}
Assume $r\geq3$.  Let $h_1,h_2,h_3,h_4\in K\setminus\{0\}$ satisfy
$\Sigma:=h_1+h_2+h_3+h_4\neq0$, and suppose that $\Sigma$ occurs among the four
holes.
Then the holes can be ordered as
\[
    p_1,p_2,q_1,q_2
\]
and there exist linearly independent characters
\[
    \kappa,\eta,\mu\in\whK
\]
such that, with
\[
    \theta=\kappa+\eta,
    \qquad
    \nu=\kappa+\mu,
\]
one has
\begin{equation}\label{eq:character-choice-conditions}
\langle\kappa,\Sigma\rangle=1,
\quad
\langle\eta,p_1\rangle=1,
\quad
\langle\theta,p_2\rangle=1,
\quad
\langle\mu,q_1\rangle=1,
\quad
\langle\nu,q_2\rangle=1.
\end{equation}
\end{lemma}

\begin{proof}
Remove one occurrence of $\Sigma$.  The remaining three nonzero holes have sum
zero, so they are
\[
 a,\ b,\ a+b
\]
for linearly independent $a,b\in K$.  There are only two cases.

Suppose first that $\Sigma\notin\Span\{a,b\}$.  Order the holes as
\[
 p_1=a+b,\qquad p_2=a,\qquad q_1=b,\qquad q_2=\Sigma.
\]
Choose the coordinate characters on the independent triple $\Sigma,a,b$ so that
\[
 \kappa(\Sigma)=1,\quad
 \eta(a)=1,\quad
 \mu(b)=1,
\]
and each of $\kappa,\eta,\mu$ vanishes on the other two vectors.  Then the
three characters are independent, $\kappa(\Sigma)=1$, and
\[
 \eta(p_1)=1,\quad
 (\kappa+\eta)(p_2)=1,\quad
 \mu(q_1)=1,\quad
 (\kappa+\mu)(q_2)=1.
\]

It remains to consider $\Sigma\in\Span\{a,b\}$.  Relabel $a,b$ inside their
two-dimensional span so that $\Sigma=a+b$; the hole multiset is then
\[
 \{a,b,a+b,a+b\}.
\]
Choose $c\notin\Span\{a,b\}$ and let $\kappa,\eta,\lambda$ be the coordinate
characters dual to $a,b,c$, respectively.  Put
\[
 p_1=a+b,\qquad p_2=a,\qquad q_1=a+b,\qquad q_2=b,
 \qquad
 \mu=\eta+\lambda.
\]
Now $\kappa,\eta,\mu$ are independent and direct evaluation gives all five
conditions in \eqref{eq:character-choice-conditions}.
\end{proof}

\begin{lemma}[Nonaffine support orbit]\label{lem:nonaffine-filter-support}
Let $p_1,p_2,q_1,q_2\in K\setminus\{0\}$.  Let
$\kappa,\eta,\mu$ be independent and put
\[
    \theta=\kappa+\eta,
    \qquad
    \nu=\kappa+\mu.
\]
Suppose
\[
\langle\eta,p_1\rangle=1,
\quad
\langle\theta,p_2\rangle=1,
\quad
\langle\mu,q_1\rangle=1,
\quad
\langle\nu,q_2\rangle=1.
\]
Define
\[
F=\Omega_{\eta,p_1}\wedge\Omega_{\theta,p_2},
\qquad
G=\Omega_{\mu,q_1}\wedge\Omega_{\nu,q_2},
\]
and write
\[
F=\sum_{Q\in\binom{\whK}{4}} f_Qe_Q,
\qquad
G=\sum_{Q\in\binom{\whK}{4}} g_Qe_Q.
\]
Put
\[
Q_0=\{0,\eta,\mu,\kappa+\eta+\mu\}.
\]
Then $\operatorname{supp}(F)\cap\operatorname{supp}(G)$ is exactly the
translation orbit
\[
    \{Q_0+\xi:\xi\in\whK\}.
\]
The orbit has size $N$, and there is a nonzero constant $c$, independent of
$\xi$, such that
\begin{equation}\label{eq:nonaffine-filter-product}
    f_{Q_0+\xi}g_{Q_0+\xi}
    =
    c\eps_\xi(p_1+p_2+q_1+q_2).
\end{equation}
Moreover,
\[
    \sum_{\rho\in Q_0}\rho=\kappa.
\]
\end{lemma}

\begin{proof}
By Lemma~\ref{lem:edge-filter}, a common-support four-set has one
decomposition into an $\eta$-pair and a $\theta$-pair, and another into a
$\mu$-pair and a $\nu$-pair.  Since
\[
    \eta+\theta=\mu+\nu=\kappa
\]
and $\kappa,\eta,\mu$ are independent, tracing the resulting alternating
four-cycle shows that the set is precisely
\[
    Q_0+\xi
    =
    \{\xi,\xi+\eta,\xi+\mu,
      \xi+\kappa+\eta+\mu\}.
\]
Conversely, each such translate has the decompositions
\[
\{\xi,\xi+\eta\}
\sqcup
\{\xi+\mu,\xi+\mu+\theta\}
\]
and
\[
\{\xi,\xi+\mu\}
\sqcup
\{\xi+\eta,\xi+\eta+\nu\}.
\]
The decompositions are unique, so Lemma~\ref{lem:edge-filter} also shows
that both coefficients at $Q_0$ are nonzero.  Thus the displayed translates
are exactly the common support.

\[
    \sum_{\rho\in Q_0}\rho=\kappa\neq0.
\]
Hence $Q_0$ is nonaffine and has no nonzero translation stabilizer, so its
orbit has size $N$.  Finally put $c=f_{Q_0}g_{Q_0}\neq0$.
Lemma~\ref{lem:shift-covariance} multiplies the two coefficients under
translation by $\xi$ by $\eps_\xi(p_1+p_2)$ and
$\eps_\xi(q_1+q_2)$, respectively.  The common orientation sign occurs
twice and cancels in their product.  This gives
\eqref{eq:nonaffine-filter-product}.
\end{proof}

\begin{proposition}[Nonaffine boundary at the total-sum hole]
\label{prop:four-holes-nonzero-sum}
Assume $r\geq3$.  Let $h_1,h_2,h_3,h_4\in K\setminus\{0\}$ have nonzero
sum
\[
 \Sigma=h_1+h_2+h_3+h_4,
\]
assume that $\Sigma$ occurs among the four holes, and let $\mathbf m$ be any
profile of size $N-4$ satisfying $\sigma(\mathbf m)=\Sigma$.
Then, after a suitable partition of the holes into
\[
    p_1,p_2\mid q_1,q_2,
\]
there exist deleted vertex sets $A,B\subseteq K$, each of size
four, such that $A$ is the union of two disjoint edges of differences
$p_1,p_2$, $B$ is the union of two disjoint edges of differences $q_1,q_2$, and
\[
    [z^{\mathbf m}]\det C(z)_{A^c,B^c}\neq0.
\]
\end{proposition}

\begin{proof}
Choose $p_1,p_2,q_1,q_2$ and $\kappa,\eta,\mu$ by
Lemma~\ref{lem:nonzero-sum-character-choice}.  Put
\[
    \theta=\kappa+\eta,
    \qquad
    \nu=\kappa+\mu,
\]
and let $F,G,Q_0$ be as in Lemma~\ref{lem:nonaffine-filter-support}.  That
lemma gives
\[
\Phi_{\mathbf m}(F,G)
=
c\sum_{\xi\in\whK}
\eps_\xi(\sigma(\mathbf m))b_{Q_0+\xi}(\mathbf m)
\]
for some $c\neq0$.  By Lemma~\ref{lem:character-translation},
\[
    b_{Q_0+\xi}(\mathbf m)
    =
    \eps_\xi(\sigma(\mathbf m))b_{Q_0}(\mathbf m).
\]
Hence
\[
\Phi_{\mathbf m}(F,G)
=
cN b_{Q_0}(\mathbf m).
\]
Now
\[
    \sum_{\rho\in Q_0}\rho=\kappa
\]
and
\[
    \langle\kappa,\sigma(\mathbf m)\rangle=1.
\]
Corollary~\ref{cor:nonaffine-four-deleted} therefore gives
\[
    b_{Q_0}(\mathbf m)\neq0.
\]
Thus the filtered sum is nonzero, and
Proposition~\ref{prop:filtered-minor} applies.
\end{proof}

The two branches now prove the single proposition announced at the beginning
of the section.

\begin{proof}[Proof of Proposition~\ref{prop:four-hole-boundary}]
Put $\Sigma=\sigma(\mathbf m)$.  If $\Sigma$ is not one of the holes, apply
Proposition~\ref{prop:four-holes-affine}.  If $\Sigma$ is one of the holes,
apply Proposition~\ref{prop:four-holes-nonzero-sum}.  These alternatives
exhaust the possibilities, and each gives the split, the two legal endpoint
sets, and the required nonzero complementary-minor coefficient.
\end{proof}

\begin{proof}[Proof of Theorem~\ref{thm:main}\textup{(i)}]
If $s=3$, then $K=H\cong\F_2^2$ and all four requests are internal.
Admissibility and the rank-two argument in
Lemma~\ref{lem:zero-sum-character-choice} show that their multiset is
$\{p,p,p,p\}$ or $\{p,p,q,q\}$.  Split it into two pairs of equal
directions.  For every nonzero $x\in K$, the two cosets of
$\Span\{x\}$ are two disjoint $x$-edges and hence realize the pair $x,x$.
Using one such perfect matching in each affine half gives the required
solution.

Assume $s\geq4$, put
\[
 K=H\cong\F_2^r,
 \qquad
 r=s-1\geq3,
\]
and write the four internal requests as
$h_1,h_2,h_3,h_4\in K\setminus\{0\}$.  Write every crossing request as
$\alpha+d_i$ and let $\mathbf m$ be the profile of the $d_i$.  Since the
number of crossing requests is even, admissibility gives
\[
 \sigma(\mathbf m)=h_1+h_2+h_3+h_4.
\]
Proposition~\ref{prop:four-hole-boundary} supplies the split and endpoint sets
$A,B$.  The nonzero determinant coefficient gives the residual bijection by
\eqref{eq:minor-certificate}, and Lemma~\ref{lem:boundary-assembly} adds the
four internal edges.  This is the required solution.
\end{proof}

\begin{corollary}[Exactly four even-weight requests]\label{cor:four-even}
Every admissible instance with exactly four even-weight requests has a
solution.
\end{corollary}

\begin{proof}
Apply Theorem~\ref{thm:main}\textup{(i)} to the parity hyperplane
$H_{\rm even}$.
\end{proof}

The reason four holes require the fourth compound is structural.  For a
linear map $M$, its $h$th compound is the induced map on
$\bigwedge^h$; in exterior bases its entries are the $h\times h$ minors of
$M$.  Thus ``second'' and ``fourth'' refer to the exterior degrees $2$ and
$4$, not to powers of the Walsh matrix.
For two holes, deleting one internal edge on each side leads to
$(N-2)\times(N-2)$ complementary minors and the second compound of the
Walsh matrix.  Four holes require two internal edges on each side, hence
four deleted vertices on each side and the fourth compound.  The affine
and nonaffine four-character configurations are genuinely different: an
affine four-set has a two-dimensional translation stabilizer, whereas a
nonaffine four-set has trivial translation stabilizer and nonzero character
sum.  The two $2$-adic lemmas above reflect precisely this dichotomy.

The four-hole proof closes because its filters reduce the boundary search to
two possible geometries for a deleted four-character set: affine and
nonaffine.  The zero-sum eight-hole case still lies within the uniform
affine-support criterion and is treated next.  The genuinely new obstruction
occurs when the total sum is nonzero and is itself one of the holes:
quotienting by its span then creates a zero hole, so the affine reduction is
no longer legal.  Section~\ref{sec:eight-nonzero} handles this obstruction by
an invariant $8\to4$ contraction, in which a translation-invariant
eight-character support becomes a four-character support in the quotient.

\section{Eight holes with zero total sum}\label{sec:eight-zero}

Four internal edges must now be placed in each affine half.  Their directions
are prescribed, but neither the four--four split of the holes nor the endpoint
sets are known in advance.  If the holes are $h_1,\ldots,h_8$ and $\mathbf m$
is the reduced crossing profile, then admissibility gives
\[
 \sigma(\mathbf m)=h_1+\cdots+h_8.
\]
We call any profile $\mathbf m$ of size $N-8$ satisfying this identity
\emph{compatible} with the eight holes.
By Proposition~\ref{prop:filtered-minor}, the main route is to choose a split
\[
 p_1,p_2,p_3,p_4\mid q_1,q_2,q_3,q_4
\]
and legal filters $F,G$ for which
\[
 \Phi_{\mathbf m}(F,G)\neq0.
\]
The determinant--Fourier dictionary and boundary assembly were established
in Section~\ref{sec:framework}, while
Proposition~\ref{prop:filtered-minor} supplies the filtered boundary
reduction.

\begin{theorem}[Zero-sum eight-hole theorem]\label{thm:eight-hole-zero-sum}
Let $s\geq4$, let $H\leq\F_2^s$ be a hyperplane, and let $\cD$ be an
admissible request instance.  Suppose exactly eight members of $\cD$, counted
with multiplicity, lie in $H$.  If the sum of these eight holes is zero, then
$\cD$ has a solution.
\end{theorem}

When the total sum is zero, compatibility says
$\sigma(\mathbf m)=0$.  The common proof map therefore reduces the entire
boundary problem to one combinatorial choice: split the eight holes into two
groups of four and find distinct characters taking value one on the two
groups.  For $r\geq4$, Proposition~\ref{prop:uniform-affine-filter} then
forces an affine eight-character support and completes the filtered-minor
step.  At $r=3$ the same edge filters are used in top exterior degree, as
shown in the proof.  The following lemma makes exactly the needed choice.  The
resulting support is an affine three-flat, meaning a coset of a
three-dimensional character subspace.

\begin{lemma}[Four--four splitting]\label{lem:four-four-splitting}
Let $K\cong\F_2^r$ with $r\geq3$, and let
\[
    h_1,\ldots,h_8\in K\setminus\{0\}
\]
be labeled vectors satisfying
\[
    h_1+\cdots+h_8=0.
\]
Then the eight labeled vectors can be partitioned into
\[
    P=(p_1,p_2,p_3,p_4),
    \qquad
    R=(q_1,q_2,q_3,q_4),
\]
and there exist distinct characters $\eta,\theta\in\whK$ such that
\[
    \langle\eta,p_i\rangle=1
    \qquad(i=1,2,3,4)
\]
and
\[
    \langle\theta,q_i\rangle=1
    \qquad(i=1,2,3,4).
\]
If $r=3$, the split may moreover be chosen so that
\[
 p_1+p_2+p_3+p_4=q_1+q_2+q_3+q_4=0.
\]
\end{lemma}

\begin{proof}
Suppose first that $r=3$.  Partition the eight labels into inclusion-minimal
nonempty zero-sum blocks.  Such a block has size $2$, $3$, or $4$, because a
minimal dependence in a three-dimensional binary space has at most four
terms.  If a block has size four, take it as $P$.  Otherwise the block sizes
are either $2+2+2+2$ or $2+3+3$.  In the first case, the union of two
$2$-blocks is a zero-sum four-subset.  In the second, each $3$-block is the
set of nonzero points of a two-dimensional subspace.  The two subspaces meet
in a nonzero point $c$; in each block the other two labels have sum $c$, so
those four labels have sum zero.  In every case we obtain a labeled
four-subset $P$ with sum zero, and its complement $R$ also has sum zero.

For either side, the equations $\lambda(x)=1$ are consistent.  Indeed, an
odd dependence among four nonzero vectors of total sum zero would have one or
three terms; one term is impossible, and three terms would force the fourth
vector to be zero.  Let $\mathcal A_P$ and $\mathcal A_R$ be the two nonempty
solution sets.  If they contain distinct elements, we are done.  Otherwise
both are the same singleton $\{\lambda\}$.  Then all eight vectors belong to
the four-point affine hyperplane
$A=\{x\in K:\lambda(x)=1\}$.  The points of $A$ span $K$ and their only
linear relation is that their total sum is zero.  Hence the parity vector of
the four multiplicities is either all even or all odd.  If all are even,
either one value occurs at least four times or two values occur at least
twice.  If all are odd, the multiplicities are $5,1,1,1$ or $3,3,1,1$ up to
order.  Thus in every case there is a zero-sum labeled four-subset $P'$ of
the form $\{a,a,a,a\}$ or $\{a,a,b,b\}$.  Its span has dimension at most
two, so $|\mathcal A_{P'}|\geq2$; its complement $R'$ again has sum zero and
$\mathcal A_{R'}\neq\varnothing$.  Choose
$\theta\in\mathcal A_{R'}$ and then
$\eta\in\mathcal A_{P'}\setminus\{\theta\}$.  This proves the claim for
$r=3$.

Assume henceforth that $r\geq4$.

Call a labeled four-subset \emph{good} if it contains no zero-sum labeled
triple.  We first show that the eight labels can be split into two good
four-subsets.

If some labeled four-subset has zero sum, take it as $P$ and its complement
as $R$.  Both have zero sum.  A zero-sum four-subset of nonzero vectors
cannot contain a zero-sum triple, since the remaining vector would then be
zero.  Thus both sides are good.

Assume therefore that no labeled four-subset has zero sum.  Let $\mathcal T$
be the family of zero-sum labeled triples.  We claim that
\[
    |\mathcal T|\leq4.
\]
Indeed, no two distinct values can both occur at least twice, since two
copies of each would form a zero-sum four-subset.  Also no value can occur
four times.  Let $a$ be the unique repeated value if one exists; otherwise
choose any occurring value $a$.  Let its multiplicity be $m\in\{1,2,3\}$.

Every zero-sum triple containing a copy of $a$ has the form
\[
    a,b,c,
    \qquad
    b+c=a.
\]
The pairs $\{b,c\}$ occurring here are pairwise disjoint.  Let $p$ be their
number.  Each such pair gives exactly $m$ labeled zero-sum triples, one for
each labeled copy of $a$.

Let $q$ be the number of zero-sum triples avoiding $a$.  Such triples are
pairwise disjoint and are disjoint from all the $2p$ vertices occurring in
the pairs above.  Indeed, if two distinct zero-sum triples shared exactly
one label, their symmetric difference would be a zero-sum four-subset.  If
they shared two labels, their third values would be equal, contradicting the
choice of $a$ as the only possible repeated value.  The same symmetric-
difference argument shows that a zero-sum triple avoiding $a$ cannot meet
one of the pairs $\{b,c\}$ in exactly one label, while meeting it in both
would force its third value to be $a$.

Hence
\[
    2p+3q\leq8-m
\]
and
\[
    |\mathcal T|=mp+q.
\]
If $m=1$, the first inequality immediately gives $p+q\leq3$.  If $m=2$,
the only way to have $2p+q>4$ is $p=3$ and $q=0$.  In that case the six
non-$a$ labels split into three pairs of sum $a$, so the total sum of all
eight labels is
\[
    2a+3a=a\neq0,
\]
a contradiction.  If $m=3$, the only way to have $3p+q>4$ is $p=2$ and
$q=0$.  Then the four labels in the two pairs have total sum zero, leaving
one further non-$a$ label $x$.  The total sum is
\[
    3a+x=a+x,
\]
so $x=a$, contradicting that the multiplicity of $a$ is exactly three.
Thus $|\mathcal T|\leq4$ in all cases.

There are $\binom84=70$ labeled four-subsets.  A fixed zero-sum triple is
contained in exactly five four-subsets, and is contained in the complement
of exactly five four-subsets.  Therefore at most
\[
    10|\mathcal T|\leq40
\]
four-subsets are bad on one side or on the other.  Since $70>40$, there is
a four-subset $P$ such that both $P$ and its complement $R$ are good.

For a good labeled four-subset, list its entries as
$P=(p_1,p_2,p_3,p_4)$.  The linear system
\[
    \lambda(p_i)=1
    \qquad(i=1,2,3,4)
\]
is consistent.  Indeed, a consistency obstruction is precisely an odd
linear dependence among the $p_i$.  Since the vectors are nonzero, an odd
dependence among four vectors has size three, which would be a zero-sum
triple.  Thus there exists a character taking value one on all of $P$, and
similarly for $R$.

It remains to make the two characters distinct.  Let
\[
    \mathcal A_P
    =\{\lambda\in\whK:\lambda(p)=1\text{ for all }p\in P\}
\]
and define $\mathcal A_R$ analogously.  If we cannot choose distinct
members of $\mathcal A_P$ and $\mathcal A_R$, then both sets are the same
singleton $\{\lambda\}$.  Hence both $P$ and $R$ span $K$, so $r\leq4$.
Since $r\geq4$, we have $r=4$.  Moreover,
\[
    \lambda(h_i)=1
    \qquad(i=1,\ldots,8).
\]
Consequently every four-subset is good.

If some vector value is repeated, choose a new four-subset $P'$ containing
two labeled copies of that value.  Then $\dim\Span(P')\leq3$, so
$|\mathcal A_{P'}|\geq2$.  Since $\lambda\in\mathcal A_{(P')^c}$, choose
$\eta\in\mathcal A_{P'}\setminus\{\lambda\}$ and
$\theta=\lambda$.

If all eight values are distinct, they are exactly the eight points of the
affine hyperplane
\[
    \{x\in K:\lambda(x)=1\}.
\]
Choose a two-dimensional subspace $L\leq\ker\lambda$ and a point
$a$ with $\lambda(a)=1$, and put
\[
    P'=a+L.
\]
Then $P'$ has four elements and spans a space of dimension at most three,
so $|\mathcal A_{P'}|\geq2$.  Again take
$\eta\in\mathcal A_{P'}\setminus\{\lambda\}$ and
$\theta=\lambda$ on the complement.  This completes the proof.
\end{proof}

\begin{proof}[Proof of Theorem~\ref{thm:eight-hole-zero-sum}]
Put $K=H$ and retain $\alpha\in V\setminus K$.
If $s=4$, then $K\cong\F_2^3$, $|K|=8$, and there are no crossing
requests.  Apply Lemma~\ref{lem:four-four-splitting} in its rank-three form.
It gives a split $P=(p_1,\ldots,p_4)$, $R=(q_1,\ldots,q_4)$ with zero sum on
each side and characters $\eta,\theta$ taking value one on their respective
four directions.  Put
\[
 F=\bigwedge_{i=1}^4\Omega_{\eta,p_i}.
\]
Choose a two-dimensional complement $E$ of $\Span\{\eta\}$ in $\whK$.
Lemma~\ref{lem:affine-filter-walsh}, with $j=4$ and
$U=\Span\{\eta\}\oplus E=\whK$, shows that the coefficient of
$e_{\whK}$ in $F$ is nonzero, because $\sum_i p_i=0$.  Hence some endpoint
wedge occurring in $F$ is nonzero.  The Walsh columns form a basis, so that
wedge has eight distinct endpoints and gives a perfect matching of $K$ with
differences $p_1,\ldots,p_4$.  The same argument with $\theta$ and $R$
gives a perfect matching of the second copy of $K$.  Translate that matching
to $\alpha+K$; together the two matchings solve the instance.

Assume now that $s\geq5$.  Write
\[
 |K|=N=2^{s-1},\qquad r=s-1\geq4.
\]
Write the eight holes as $h_1,\ldots,h_8\in K\setminus\{0\}$ with total
sum zero, and write the remaining requests uniquely as
$\alpha+d_i$, $d_i\in K$, $1\leq i\leq N-8$.  Let $\mathbf m$ be their
profile.  Admissibility gives $\sigma(\mathbf m)=0$.

Apply Lemma~\ref{lem:four-four-splitting}.  We obtain a split
$P=(p_1,\ldots,p_4)$, $R=(q_1,\ldots,q_4)$ and distinct characters
$\eta,\theta$ taking value one on their corresponding four holes.  Put
\[
 d=p_1+p_2+p_3+p_4=q_1+q_2+q_3+q_4.
\]
Then $\eta(d)=\theta(d)=0$ and
$\dim d^\perp\geq r-1\geq3$.  Proposition
~\ref{prop:uniform-affine-filter}, with $j=4$ and
$\Sigma=\sigma(\mathbf m)=0$, supplies two eight-element deleted sets $A,B$ and
a nonzero coefficient
\[
 [z^{\mathbf m}]\det C(z)_{A^c,B^c}\neq0.
\]
By~\eqref{eq:minor-certificate} this supplies the residual crossing
bijection, and Lemma~\ref{lem:boundary-assembly} completes the solution.
\end{proof}

This completes the zero-sum case.  When the total hole sum is nonzero, the
same affine criterion still works if that sum is not a hole, after quotienting
by its span.  The next section begins with that case and then treats the
genuine obstruction in which the total sum is itself a hole.

\section{Eight holes with nonzero total sum}\label{sec:eight-nonzero}

Let the holes be $h_1,\ldots,h_8$, let $\mathbf m$ be a compatible crossing
profile of size $N-8$, and put
\[
 \Sigma:=\sigma(\mathbf m)=h_1+\cdots+h_8\neq0.
\]
This section proves the complementary eight-hole case.

\begin{theorem}[Nonzero-sum eight-hole theorem]
\label{thm:eight-hole-nonzero-sum}
Let $s\geq4$, let $H\leq\F_2^s$ be a hyperplane, and let $\cD$ be an
admissible request instance.  Suppose exactly eight members of $\cD$, counted
with multiplicity, lie in $H$, and that their sum is nonzero.  Then $\cD$ has
a solution.
\end{theorem}

The location of $\Sigma$ gives the top-level division.  If $\Sigma$ is not one of the
eight holes, quotienting by $\Span\{\Sigma\}$ reduces the tuple to the zero-sum
four--four splitting lemma.  If $\Sigma$ is a hole, that quotient contains a zero
hole and is no longer a legal reduction.  We then force an invariant
$8\to4$ contraction through a paired lift.  Relation spaces identify the
configurations with no such lift, and exact projectors handle only that fixed
exceptional list.  All remaining machinery has a single purpose: to replace
the illegal quotient when $\Sigma$ itself is a hole.

\subsection{When the total sum is not a hole: the affine quotient}

\begin{proposition}[Generic eight-hole case]\label{prop:generic-eight}
 Assume $r\geq4$.  Let
\[
    h_1,\ldots,h_8\in K\setminus\{0\}
\]
have nonzero sum, and assume
\[
    h_i\neq h_1+\cdots+h_8
    \qquad(i=1,\ldots,8).
\]
Let $\mathbf m$ be a profile of size $N-8$ with
\[
    \sigma(\mathbf m)=h_1+\cdots+h_8.
\]
Then there are deleted sets $A,B\subseteq K$, each of size eight, realizing
a four--four split of the holes and satisfying
\[
    [z^{\mathbf m}]\det C(z)_{A^c,B^c}\neq0.
\]
\end{proposition}

\begin{proof}
Put $\Sigma=\sigma(\mathbf m)$ and pass to the quotient
\[
    \overline K=K/\Span\{\Sigma\}.
\]
The images $\overline h_i$ are nonzero, because no hole equals $\Sigma$, and
their total sum is zero.  Since $\dim\overline K=r-1\geq3$,
Lemma~\ref{lem:four-four-splitting}, applied in the quotient, gives a split
\[
    P=(p_1,p_2,p_3,p_4),
    \qquad
    R=(q_1,q_2,q_3,q_4),
\]
and distinct quotient characters $\eta,\theta$.  Lift them to
$\Sigma^\perp\subseteq\whK$.  Then
\[
    \langle\eta,p_i\rangle=1,
    \qquad
    \langle\theta,q_i\rangle=1
    \qquad(1\leq i\leq4).
\]
Put $d=p_1+p_2+p_3+p_4$.  Since both characters annihilate $\Sigma$ and take
value one on their four assigned holes,
\[
    \Span\{\eta,\theta\}\subseteq \Sigma^\perp\cap d^\perp.
\]
If $r=4$, then $\dim\overline K=3$, and the rank-three clause of
Lemma~\ref{lem:four-four-splitting} gives $\bar d=0$.  Hence
$d\in\Span\{\Sigma\}$ and
\[
 \Sigma^\perp\cap d^\perp=\Sigma^\perp,
 \qquad \dim(\Sigma^\perp\cap d^\perp)=3.
\]
If $r\geq5$, then instead
\[
 \dim(\Sigma^\perp\cap d^\perp)\geq r-2\geq3.
\]
Thus Proposition~\ref{prop:uniform-affine-filter}, with $j=4$, gives the
required deleted sets in every case.
\end{proof}

\subsection{When the total sum is a hole}\label{sec:eight-sum-hole}

Here the direct quotient fails because one hole would become zero.  The
replacement proceeds in three stages: first an invariant
$8\to4$ contraction turns eight deleted characters into four quotient
characters; next paired-lift filters force precisely that geometry; finally a
relation-space calculation and exact projectors settle the configurations for
which no paired lift exists.

\subsubsection{Invariant contraction and paired lifts}
The residual four-deleted valuation is already
Corollary~\ref{cor:nonaffine-four-deleted}.  This is the first point at which
the four-hole arithmetic becomes a direct input to the eight-hole proof.

Let $u\in\whK\setminus\{0\}$.  If $Q=Q+u$ and $|Q|=8$, then $Q$ is a
union of four $u$-pairs.  Write
\[
 \overline{\whK}=\whK/\Span\{u\},
 \qquad
 \bar Q=Q/\Span\{u\}.
\]
If $u$ annihilates $\sigma(\mathbf m)$, pairing an element of
$\overline{\whK}$ with $\sigma(\mathbf m)$ is
well-defined.

To see why an $8\to4$ contraction appears, let
$u\in\whK\setminus\{0\}$ and choose
$\rho_1,\ldots,\rho_4$ in distinct cosets modulo $\Span\{u\}$.  Put
\[
    Q=\bigcup_{i=1}^4\{\rho_i,\rho_i+u\}.
\]
Then $Q$ has eight elements and satisfies $Q=Q+u$, whereas
$\bar Q=Q/\Span\{u\}$ has four elements.  Pairing the Walsh rows
indexed by $\rho$ and $\rho+u$ extracts a forced power of two and leaves a
four-deleted Walsh permanent on the quotient.  The next lemma makes this
picture invariant and supplies the exact valuation needed by the filters.

\begin{lemma}[Invariant eight-to-four contraction]\label{lem:invariant-eight-to-four}
Assume $r\geq4$.  Let $0\neq u\in\whK$, and let
$Q\subseteq\whK$ have eight elements and satisfy
$Q=Q+u$.  Let $|\mathbf m|=N-8$ and assume
\[
    \langle u,\sigma(\mathbf m)\rangle=0,
    \qquad
    \left\langle\sum_{\bar\rho\in\bar Q}\bar\rho,
    \sigma(\mathbf m)\right\rangle=1.
\]
Then
\[
 \vTwo\bigl(b_Q(\mathbf m)\bigr)
 =
 \vTwo\left(\frac{(N-8)!}{\mathbf m!}\right)-1.
\]
In particular $b_Q(\mathbf m)\neq0$.
\end{lemma}

\begin{proof}
Let $n=N-8$ and use labeled columns $x_1,\ldots,x_n\in K$ with profile
$\mathbf m$.  Pair the rows of $W_{\mathbf m}^{Q}$ as
$\{\rho,\rho+u\}$.  For two labeled columns $x_a,x_b$, the local
$2\times2$ permanent is
\begin{align*}
&\eps_\rho(x_a)\eps_{\rho+u}(x_b)
 +\eps_\rho(x_b)\eps_{\rho+u}(x_a)\\
&\qquad=
 \eps_\rho(x_a+x_b)
 \bigl(\eps_u(x_a)+\eps_u(x_b)\bigr).
\end{align*}
It vanishes unless $u(x_a)=u(x_b)$.  In a surviving pair it is twice a
sign.  Since $u$ annihilates $\sigma(\mathbf m)$, both $u$-parity classes of the labeled columns have
even cardinality.  Thus every surviving term pairs columns within the two
parity classes and contributes a common factor $2^{n/2}$.

Each permanent term determines a unique such column pairing together with a
bijection of its pairs to the quotient rows.  Hence this regrouping neither
omits nor repeats a term; within each parity class the extracted sign is fixed
by the pair sums.

For a surviving labeled pairing $\mathcal P$, replace every column pair by
its sum.  These sums lie in $u^\perp$, their total sum is still
$\sigma(\mathbf m)$, and the
residual row set is
\[
 \overline{\whK}\setminus\bar Q.
\]
Hence the residual permanent is a four-deleted Walsh permanent of order
$N/2-4$.  By Corollary~\ref{cor:nonaffine-four-deleted}, every residual
permanent has exact valuation
\[
 \vTwo((N/2-4)!)-1.
\]
Let $n_i$ be the size of the $i$th $u$-parity class.  The number of
surviving labeled column pairings is
\[
 (n_0-1)!!(n_1-1)!!,
\]
where we use the convention $(-1)!!=1$ when a class is empty.  This number is
odd.
After division by the common power of two, every residual permanent is odd,
so an odd number of them cannot cancel modulo two.  Therefore
\[
 \vTwo\bigl(\per W_{\mathbf m}^{Q}\bigr)
 =\frac n2+\vTwo((N/2-4)!)-1.
\]
Since $n=2(N/2-4)$,
\[
 \vTwo((N-8)!)
 =\frac n2+\vTwo((N/2-4)!),
\]
and \eqref{eq:coefficient-permanent} proves the formula after division by
$\mathbf m!$.
\end{proof}

\medskip
The contraction becomes useful once the filters force a common support to be
invariant in a character direction $u$.  A paired-lift datum records exactly
the four--four split and character incidences needed to create that support.
The resulting criterion handles all but a finite collection of the remaining
configurations in which the total sum is itself a hole.

\begin{definition}[Paired-lift datum]\label{def:paired-lift}
For a labeled eight-hole multiset and a compatible profile $\mathbf m$, a
paired-lift datum consists of characters $u,\beta,\gamma\in\whK$ and a labeled split
\[
 P=(p_1,p_2,p_3,p_4),
 \qquad
 R=(q_1,q_2,q_3,q_4)
\]
of the eight holes such that
\[
 \langle u,\sigma(\mathbf m)\rangle=0,
 \qquad
 \langle u,p_i\rangle=1,
 \qquad
 \langle u,q_i\rangle=0,
\]
\[
 \langle\beta,q_1\rangle=\langle\beta,q_2\rangle=1,
 \qquad
 \langle\gamma,q_3\rangle=\langle\gamma,q_4\rangle=1,
\]
and, with
\[
 d=p_1+p_2+p_3+p_4,
 \qquad
 \kappa=\beta+\gamma,
\]
one has
\[
 \langle\kappa,\sigma(\mathbf m)\rangle=\langle\kappa,d\rangle=1.
\]
\end{definition}

Given such a datum, define
\[
 F=\bigwedge_{i=1}^4\Omega_{u,p_i},
 \qquad
 G=
 \Omega_{\beta,q_1}\wedge\Omega_{\beta,q_2}
 \wedge\Omega_{\gamma,q_3}\wedge\Omega_{\gamma,q_4}.
\]
The incidence assumptions make all eight edge filters legal.

The contraction lemma controls the deleted Walsh coefficient once a common
support is known to be $u$-invariant.  One issue remains: the contributions
of the many such supports could cancel in $\Phi_{\mathbf m}(F,G)$.  The next
exact finite lemma verifies the invariant $8\to4$ geometry and computes the
local translation signature needed to prevent that cancellation.

\begin{exactlemma}[Local paired-filter signature]\label{lem:paired-local-signature}
For every paired-lift datum, the characters $u,\beta,\gamma$ are linearly
independent.  Put
\[
 D=\Span\{u,\beta,\gamma\}\cong\F_2^3.
\]
After removing the common nonzero scalar $(2N)^8$ from the product of the
two filter coefficients, the following hold.
\begin{enumerate}[label=\textnormal{(\roman*)}]
 \item every common-support set $Q$ is $u$-invariant;
 \item every common-support set occupies exactly two distinct $D$-cosets,
 each in four vertices;
 \item for $\bar Q=Q/\Span\{u\}$,
 \[
   \sum_{\bar\rho\in\bar Q}\bar\rho=\bar\kappa;
 \]
 \item fix an unordered pair of distinct ambient $D$-cosets, choose an
 ordering $x+D,y+D$, and put $\delta=x+y$.  Use the ordered basis
 $(u,\beta,\gamma)$ to identify their union with $D\times\F_2$ by, for
 $t\in D$,
 \[
   x+t\longleftrightarrow(t,0),
   \qquad
   y+t\longleftrightarrow(t,1).
 \]
 The group $\mathcal G=D\times\F_2$ acts by
 \[
   (a,e)\cdot(t,j)=(t+a,j+e),
 \]
 and carries the twist
 \[
    \chi_{\mathbf m}(a,e)
    =
    (-1)^{\langle a,\sigma(\mathbf m)\rangle
    +e\langle\delta,\sigma(\mathbf m)\rangle}.
 \]
 If $F=\sum_Qf_Qe_Q$ and $G=\sum_Qg_Qe_Q$, put
 $c_Q=(2N)^{-8}f_Qg_Q$.  For every $\mathcal G$-orbit $\mathcal O$ of
 common supports, encode a support by its $16$-bit incidence mask in the
 following fixed order.  Identify $D$ with $\F_2^3$ by
 $u=001$, $\beta=010$, and $\gamma=100$.  If
 $t=au+b\beta+c\gamma$ and $e\in\F_2$, then $(t,e)$ occupies bit position
 \[
  a+2b+4c+8e.
 \]
 Thus bits $0,\ldots,7$ encode the first $D$-coset and bits $8,\ldots,15$
 encode the second.  Order the masks numerically and let $Q_{\mathcal O}$ be the
 least support in this order.
 For $Q\in\mathcal O$, choose $\tau_Q\in\mathcal G$ with
 $\tau_QQ=Q_{\mathcal O}$ and define
 \[
    \mathcal H(\mathcal O)
    =
    \sum_{Q\in\mathcal O}\chi_{\mathbf m}(\tau_Q)c_Q.
 \]
 The transported summand is independent of the choice of $\tau_Q$.  Define
 the twisted signature by
 \[
    \mathcal H=\sum_{\mathcal O}\mathcal H(\mathcal O)e_{\mathcal O}
 \]
 where $e_{\mathcal O}$ is the standard coordinate vector indexed by the
 orbit $\mathcal O$.  The vector $\mathcal H$ is supported on exactly one
 orbit, whose coefficient has magnitude $128$
 or $384$ and hence exact valuation $7$.
\end{enumerate}
\end{exactlemma}

\begin{proof}[Exact verification]
The incidence conditions make $u,\beta,\gamma$ nonzero.  Since
$\langle\kappa,\sigma(\mathbf m)\rangle=1$, we have $\beta\neq\gamma$.
Also $u\neq\beta$
because $\langle u,q_1\rangle=0$ whereas $\langle\beta,q_1\rangle=1$, and
$u\neq\gamma$ because $\langle u,q_3\rangle=0$ whereas
$\langle\gamma,q_3\rangle=1$.  If the three characters were dependent, then
three distinct nonzero vectors over $\F_2$ would satisfy
$u=\beta+\gamma=\kappa$, contradicting
$\langle u,\sigma(\mathbf m)\rangle=0$ and
$\langle\kappa,\sigma(\mathbf m)\rangle=1$.  Thus they are linearly independent.

Every filter direction belongs to $D=\Span\{u,\beta,\gamma\}$, so every
edge allocation stays inside a $D$-coset.  A support of $F$ is a union of
four $u$-pairs.  If a common support met at least three $D$-cosets, one
occupied coset would contain exactly one $u$-pair.  The same two vertices
would then have to form a $\beta$- or $\gamma$-edge in an allocation
contributing to $G$, forcing $u=\beta$ or $u=\gamma$, a contradiction.
Thus a common support meets at most two $D$-cosets.

It cannot meet only one.  In that case $Q=x+D$.  Choose a complement of
$\Span\{u\}$ in $D$.  Lemma~\ref{lem:affine-filter-walsh}, with $j=4$,
shows that the coefficient of $F$ can be nonzero only if $d\in D^\perp$.  But
$\kappa\in D$ and $\langle\kappa,d\rangle=1$.  Hence this coefficient
vanishes.  Therefore every common support meets exactly two $D$-cosets; the
preceding single-pair argument then forces four vertices in each.

It remains to reduce the universal assertion to the finite calculation.
Record a hole only through its evaluations on the ordered basis
$(u,\beta,\gamma)$ of $D$.  For each $p_i$, its $u$-evaluation is fixed to
$1$ and its $\beta$- and $\gamma$-evaluations are free, giving eight bits.
For $q_1,q_2$, the $u$-evaluation is $0$ and the $\beta$-evaluation is $1$,
leaving two free $\gamma$-bits; for $q_3,q_4$, the $u$-evaluation is $0$ and
the $\gamma$-evaluation is $1$, leaving two free $\beta$-bits.  Thus there
are $2^{12}=4096$ raw local evaluation patterns.  The two parity conditions
$\langle\kappa,\sigma(\mathbf m)\rangle=1$ and
$\langle\kappa,d\rangle=1$ retain exactly $1024$ of them.

For two distinct ambient $D$-cosets, all remaining dependence is through the
eight bits $\langle\delta,p_i\rangle$ and
$\langle\delta,q_i\rangle$, where $\delta$ is their difference.  Enumerating
all $2^8=256$ assignments is stronger than restricting to assignments
realizable by a particular ambient $\delta$.  The normalized filter
coefficients, common supports, translation twist, and quotient sum depend
only on these twelve local bits and eight external bits: inside either coset
each filter coefficient is the sign determined by the corresponding
evaluation, and translating or interchanging the two cosets uses only the
displayed twist character.  Hence every paired-lift datum maps to one of the
$1024\cdot256$ enumerated cases.

The exact calculation summarized in
Table~\ref{tab:finite-verifications} and archived in the Zenodo
record~\cite{KreindelZabokritskiyVerification2026} expands the eight
normalized filters over the integers in every one of these cases.  It checks
$2{,}097{,}152$ common supports and verifies the transport, quotient-sum,
support, and valuation assertions.  The Zenodo record contains the source program,
recorded output, and full reproduction protocol.
\end{proof}

\begin{proposition}[Paired-lift criterion]\label{prop:paired-lift}
Let $K=\F_2^r$, let $N=|K|$, and assume $r\geq4$.  Let
$P=(p_1,\ldots,p_4)$ and $R=(q_1,\ldots,q_4)$ be a labeled split of eight
nonzero vectors with nonzero total sum, and let $\mathbf m$ be a profile with
$|\mathbf m|=N-8$ and
\[
\sigma(\mathbf m)=\sum_{i=1}^4p_i+\sum_{i=1}^4q_i.
\]
If $P\mid R$ admits a paired-lift datum for $\mathbf m$, then there
are eight-element sets $A,B\subseteq K$ such that $A$ is the union of four
disjoint edges with differences $p_1,\ldots,p_4$, $B$ is the union of four
disjoint edges with differences $q_1,\ldots,q_4$, and
\[
 [z^{\mathbf m}]\det C(z)_{A^c,B^c}\neq0.
\]
\end{proposition}

\begin{proof}
By Lemma~\ref{lem:paired-local-signature}, every common-support $Q$ satisfies
$Q=Q+u$ and
\[
 \left\langle\sum_{\bar\rho\in\bar Q}\bar\rho,
 \sigma(\mathbf m)\right\rangle
 =\langle\kappa,\sigma(\mathbf m)\rangle=1.
\]
Lemma~\ref{lem:invariant-eight-to-four} therefore gives the same exact
$2$-adic valuation for every deleted coefficient $b_Q(\mathbf m)$ appearing
in the filtered sum.

Put $M=|\whK/D|=2^{r-3}$.  For every nonzero direction
$\delta\in\whK/D$, there are exactly $M/2=2^{r-4}$ unordered pairs of
$D$-cosets of difference $\delta$.  All pairs with the same direction are
translates of one another.  Shift covariance multiplies the product of the
two filter coefficients by $\eps_\zeta(\sigma(\mathbf m))$, while
Lemma~\ref{lem:character-translation} multiplies the deleted coefficient by
the same sign.  Their product is therefore unchanged.  Hence the contribution
of a fixed direction has the common factor $2^{r-4}$.

After this factor, the filter scalar, and the common deleted-coefficient
valuation are removed, Lemma~\ref{lem:paired-local-signature} says that every
direction contributes an odd integer.  There are $M-1$ nonzero directions,
and this number is odd.  The normalized filtered sum is therefore odd and in
particular nonzero.  Proposition~\ref{prop:filtered-minor} gives the
required sets $A,B$.
\end{proof}

\subsubsection{The exceptional relation spaces}
Assume from now on that the nonzero total hole sum occurs among the eight
holes and that no paired-lift datum exists.  This is exactly the branch not
settled by the affine quotient or by the invariant $8\to4$ contraction.  The
relation space replaces its many coordinate realizations by a fixed list of
linear-equivalence types.

Let
$e_1,\ldots,e_8$ be the
standard basis of $\F_2^8$, and put
$\mathbf1=e_1+\cdots+e_8=(1,\ldots,1)$.  For the labeled hole tuple
$h_1,\ldots,h_8$, define
\[
 \phi:\F_2^8\longrightarrow\Span\{h_1,\ldots,h_8\},
 \qquad
 \phi(e_i)=h_i,
\]
and put
\[
 \mathcal R=\ker\phi.
\]
The labeled configuration is determined up to linear equivalence by
$\mathcal R$.
The conditions that every hole and the total sum $\phi(\mathbf1)$ are
nonzero are
\[
 e_i\notin\mathcal R,
 \qquad
 \mathbf1\notin\mathcal R.
\]
Furthermore
\[
 \phi(\mathbf1)=h_i
 \quad\Longleftrightarrow\quad
 \mathbf1+e_i\in \mathcal R.
\]
Via their eight evaluation bits, characters of the hole span are identified
with $\mathcal R^\perp\leq\F_2^8$, where the orthogonal complement is taken
with respect to the standard dot product.  If
$\phi(\mathbf1)$ occurs among the holes, set
\[
 m_{\rm tot}=|\{i:h_i=\phi(\mathbf1)\}|.
\]
Then $1\leq m_{\rm tot}\leq6$.  Indeed, eight copies of a vector have total
sum zero, while seven copies together with an eighth hole $x$ have total sum
$\phi(\mathbf1)+x$; equality with $\phi(\mathbf1)$ would force $x=0$, which
is forbidden.

In this language, Definition~\ref{def:paired-lift} is a finite condition on
$\mathcal R^\perp$: one seeks $u\in\mathcal R^\perp$ of Hamming weight
four, pairs the four zero coordinates of $u$, and seeks
$\beta,\gamma\in\mathcal R^\perp$ taking value
one on the two selected pairs such that, for $\kappa=\beta+\gamma$,
\[
 \langle\kappa,\mathbf1\rangle=1,
 \qquad
 \langle\kappa,u\rangle=1.
\]
The second equality is exactly $\kappa(d)=1$.

\begin{exactlemma}[Exact failure classification]\label{lem:failure-classification}
Among all nonzero-sum eight-hole relation spaces for which the total sum occurs as a
hole, the paired-lift criterion fails in exactly $23$ orbits under
permutation of the eight labels.  Their multiplicity distribution is
\[
\begin{array}{c|rrrrrr}
 m_{\rm tot}&1&2&3&4&5&6\\
\hline
 \#\text{ failure orbits}&7&6&5&3&1&1.
\end{array}
\]
The $23$ orbit representatives are exactly the hole multisets listed in
\hyperref[app:certificates]{Appendix~\ref*{app:certificates}}.
\end{exactlemma}

\begin{proof}[Exact exhaustive verification]
The relation-space check in Table~\ref{tab:finite-verifications} enumerates
the subspaces in reduced row-echelon form and applies the stated tests by
exact row reduction and orthogonality.  It yields exactly the $23$
label-permutation orbits and multiplicity distribution in the statement.
The exhaustive procedure, recorded output, and orbit representatives are
archived in the Zenodo
record~\cite{KreindelZabokritskiyVerification2026}.
\end{proof}

\medskip
To turn the classification into nonzero filtered minors, we now describe the
finite certificate language for the paired-lift
failures.  Let $\Lambda\leq\whK$ be a four-dimensional character subspace and
fix an ordered basis identifying $\Lambda$ with $\F_2^4$; the symbols
$<$ below refer to the order of the four-bit coordinate masks
$0,1,\ldots,15$, read as binary integers.  For a legal
edge filter we
remove the common scalar $-2N$ and use the normalized local form
\[
 \widehat\Omega_{\alpha,p}
 =\sum_{\substack{\rho\in\Lambda\\\rho<\rho+\alpha}}\eps_\rho(p)
   e_\rho\wedge e_{\rho+\alpha},
  \qquad \alpha\in\Lambda\setminus\{0\}.
\]

A witness
\[
 \mathcal W=(P;\boldsymbol\alpha\mid R;\boldsymbol\beta)
\]
consists of ordered hole four-tuples $P=(p_1,\ldots,p_4)$ and
$R=(q_1,\ldots,q_4)$ and ordered character four-tuples
$\boldsymbol\alpha=(\alpha_1,\ldots,\alpha_4)$ and
$\boldsymbol\beta=(\beta_1,\ldots,\beta_4)$ in $\Lambda^4$.  It is legal when
\[
 \alpha_i(p_i)=\beta_i(q_i)=1.
\]
Write
\[
 \widehat F_{\mathcal W}
 =\bigwedge_{i=1}^4\widehat\Omega_{\alpha_i,p_i},
 \qquad
 \widehat G_{\mathcal W}
 =\bigwedge_{i=1}^4\widehat\Omega_{\beta_i,q_i}.
\]
The corresponding ambient filters are
\[
 F_{\mathcal W}
 =\bigwedge_{i=1}^4\Omega_{\alpha_i,p_i},
 \qquad
 G_{\mathcal W}
 =\bigwedge_{i=1}^4\Omega_{\beta_i,q_i}.
\]

The next condition is designed to prevent a common ambient support from
spreading over several $\Lambda$-cosets before the projector identity is
applied.

\begin{definition}[Strongly $\Lambda$-connected witness]
\label{def:strong-lambda-connected}
Write $[4]=\{1,2,3,4\}$.  For a nonempty $I\subseteq[4]$, put
\[
 \widehat F_I
 =\bigwedge_{i\in I}\widehat\Omega_{\alpha_i,p_i},
 \qquad
 \widehat G_I
 =\bigwedge_{i\in I}\widehat\Omega_{\beta_i,q_i}.
\]
The wedge factors are taken in increasing order of their indices.
A legal witness is strongly $\Lambda$-connected if, for every
$k\in\{2,3,4\}$, every two ordered partitions
\[
 [4]=I_1\sqcup\cdots\sqcup I_k
     =J_1\sqcup\cdots\sqcup J_k
\]
into nonempty blocks, and every bijection $\pi:[k]\to[k]$, there is an index
$t$ such that
\[
 \operatorname{supp}(\widehat F_{I_t})
 \cap
 \operatorname{supp}(\widehat G_{J_{\pi(t)}})
 =\varnothing.
\]
\end{definition}

Indeed, a common ambient support meeting $k$ distinct $\Lambda$-cosets would
allocate the left and right factors into just such matched blocks, with
intersecting local support for every matched pair.  The definition therefore
excludes supports spread across several $\Lambda$-cosets before any
cancellation between factor allocations.  It is also the finite condition
checked by the certificate verifier archived in the Zenodo
record~\cite{KreindelZabokritskiyVerification2026}.

Fix eight nonzero holes and a compatible profile $\mathbf m$ with nonzero
$\sigma(\mathbf m)$.  For an ambient witness $\mathcal W$
write
\[
 F_{\mathcal W}=\sum_Q f_Q^{\mathcal W}e_Q,
 \qquad
 G_{\mathcal W}=\sum_Q g_Q^{\mathcal W}e_Q,
\]
where $|\mathbf m|=N-8$.  The filtered scalar is
$\Phi_{\mathbf m}(F_{\mathcal W},G_{\mathcal W})$.
Let $\mathfrak O_\Lambda=\binom{\Lambda}{8}/\Lambda$ be the set of
$\Lambda$-translation orbits.  Encode an eight-set by its $16$-bit incidence
mask in the binary order on $\Lambda$ fixed above, and order these masks
numerically.  For each $\mathcal O\in\mathfrak O_\Lambda$, choose its least
member $Q_{\mathcal O}$ in this order.  If
\[
 \widehat F_{\mathcal W}=\sum_Q\widehat f_Qe_Q,
 \qquad
 \widehat G_{\mathcal W}=\sum_Q\widehat g_Qe_Q,
\]
put
\[
 \mathcal H_{\mathcal W}(\mathcal O)
 =\sum_{\zeta\in\Lambda}\eps_\zeta(\sigma(\mathbf m))
   \widehat f_{Q_{\mathcal O}+\zeta}
   \widehat g_{Q_{\mathcal O}+\zeta},
 \qquad
 \mathcal H_{\mathcal W}\in\mathbb Z^{\mathfrak O_\Lambda},
\]
and let $e_{\mathcal O}$ denote the corresponding standard coordinate
vector.  This twisted translation sum, rather than the untwisted orbit sum,
is the coordinate used in the certificate table.

\begin{lemma}[Strong coordinate-projector criterion]\label{lem:strong-projector}
Let $\mathcal W_1,\ldots,\mathcal W_t$ be strongly $\Lambda$-connected legal
witnesses for the same hole multiset and compatible profile $\mathbf m$ with
$\sigma(\mathbf m)\neq0$.  Suppose
every common support $Q$ of $F_{\mathcal W_j}$ and
$G_{\mathcal W_j}$ satisfies
\[
 \sum_{\rho\in Q}\rho=\kappa_j,
 \qquad
 \langle\kappa_j,\sigma(\mathbf m)\rangle=1,
\]
and suppose integers $c_1,\ldots,c_t$ satisfy
\[
 \sum_{j=1}^t c_j \mathcal H_{\mathcal W_j}
   =c_{\mathrm{proj}}\,e_{\mathcal O}
\]
 for one free $\Lambda$-translation orbit $\mathcal O$ (that is, an orbit
 with trivial translation stabilizer) and one nonzero integer
 $c_{\mathrm{proj}}$.
Then at least one filtered sum
$\Phi_{\mathbf m}(F_{\mathcal W_j},G_{\mathcal W_j})$ is
nonzero.
\end{lemma}

\begin{proof}
Strong $\Lambda$-connectedness excludes supports meeting two distinct ambient
$\Lambda$-cosets.  If an eight-set $Q$ has a nonzero translation stabilizer
$\zeta$, then it is a union of four $\zeta$-pairs and hence
$\sum_{\rho\in Q}\rho=0$.  The hypothesis
$\left\langle\sum_{\rho\in Q}\rho,\sigma(\mathbf m)\right\rangle=1$
therefore makes every common-support
translation orbit free, so the sixteen terms in each coordinate
$\mathcal H_{\mathcal W}(\mathcal O)$ are distinct.  The local twisted identity is
copied to each of the $N/16$
ambient $\Lambda$-cosets.  By Lemmas~\ref{lem:shift-covariance} and
\ref{lem:character-translation}, translation by $\zeta$ multiplies both the
filter product and the deleted coefficient by
$\eps_\zeta(\sigma(\mathbf m))$, so the two signs cancel.  The local and
ambient orders can differ by a fixed orientation sign.  The same reordering
multiplies both $f_Q^{\mathcal W}$ and $g_Q^{\mathcal W}$ by that sign, so
their product is unchanged.  Consequently the identity is exact:
\[
 \sum_{j=1}^t c_j
 \Phi_{\mathbf m}(F_{\mathcal W_j},G_{\mathcal W_j})
 =(-2N)^8\frac N{16}\,c_{\mathrm{proj}}\,
   b_{Q_{\mathcal O}}(\mathbf m).
\]
Since the target coordinate is nonzero, $Q_{\mathcal O}$ occurs in at least
one witness $\mathcal W_j$.  For that witness
$\sum_{\rho\in Q_{\mathcal O}}\rho=\kappa_j$, which pairs nontrivially with
$\sigma(\mathbf m)$.  The final
coefficient is therefore nonzero by
Theorem~\ref{thm:nonzero-sum-power-two-deletion}, applied with $\ell=3$.
Thus the displayed
linear combination is nonzero, and not all individual filtered sums can
vanish.
\end{proof}

\subsubsection{Completion of the exceptional branch}
We now isolate exactly what remains in the total-sum-hole case.
Write
\[
 \Sigma=h_1+\cdots+h_8=\sigma(\mathbf m).
\]
Proposition~\ref{prop:paired-lift} applies whenever a paired-lift datum
exists.  Hence only configurations in which $\Sigma$ is a hole and no paired lift
exists remain.
Lemma~\ref{lem:failure-classification} places them in exactly $23$
relation-space orbits.  This classification identifies the configurations
but does not yet produce the deleted boundary sets $A$ and $B$.

To finish, for every applicable orbit we must exhibit legal strongly
$\Lambda$-connected witnesses satisfying the common-support character-sum
condition, together with integer coefficients whose local signatures isolate
one free translation orbit.  Lemma~\ref{lem:strong-projector} then forces at
least one filtered scalar $\Phi_{\mathbf m}$ to be nonzero, and
Proposition~\ref{prop:filtered-minor} produces the required deleted
boundaries.  The next lemma records this remaining finite assertion.
\hyperref[app:certificates]{Appendix~\ref*{app:certificates}} supplies the
complete $23$-entry, $45$-witness certificate data used in its exact
verification; the verifier and recorded output are archived in the Zenodo
record~\cite{KreindelZabokritskiyVerification2026}.

\begin{exactlemma}[Exact strong certificate lemma]\label{lem:finite-exceptional}
The preceding classification has $23$ paired-lift failure orbits.  Each
admits strongly $\Lambda$-connected legal witnesses in the sense of
Definition~\ref{def:strong-lambda-connected}, satisfying the
common-support character-sum condition and an integer projector identity
onto one free $\Lambda$-translation orbit.  Thus every hypothesis of
Lemma~\ref{lem:strong-projector} holds for the corresponding certificate
printed in
\hyperref[app:certificates]{Appendix~\ref*{app:certificates}}.
\end{exactlemma}

\begin{proof}[Exact verification]
The certificate verifier archived in the Zenodo
record~\cite{KreindelZabokritskiyVerification2026} reads the entries in
Appendix~\ref{app:certificates} and
directly expands their local $\Omega$-forms over the integers.  It verifies all
legality, connectedness, support-sum, valuation, and projector assertions
listed in Table~\ref{tab:finite-verifications}.  The Zenodo record also contains the
verification procedure, complete coefficient vectors, and recorded output.
\end{proof}

\begin{lemma}[Transport of the printed certificates]
\label{lem:certificate-transport}
Let a printed representative be realized in $\F_2^d$, and let a labeled
hole tuple in $K\cong\F_2^r$ have the same relation space up to relabeling.
If $d\leq r$, then the printed certificate applies to that tuple.  In
particular, every printed certificate applies when $r\geq5$, while for
$r=4$ every certificate except the unique rank-five representative applies.
\end{lemma}

\begin{proof}
Relabeling merely permutes the eight factors.  Let $L_0$ be the span of the
printed representative and $L$ the span of the given tuple.  Equality of the
relation spaces makes the assignment of printed holes to given holes extend
to an isomorphism $T:L_0\to L$.  When $d\leq r$, extend $T$ to an injection
$J:\F_2^d\hookrightarrow K$ and choose a linear left inverse
$\pi:K\to\F_2^d$.  A printed character $\alpha$ lifts to
$\widetilde\alpha=\alpha\circ\pi$.  This preserves every incidence,
wedge-support identity, translation signature, and value of
$\sum_{\rho\in Q}\rho$, so the printed projector identity transports.
All printed representatives except the unique rank-five representative are
realized in $\F_2^4$; that representative is realized in $\F_2^5$.
\end{proof}

\begin{proof}[Proof of Theorem~\ref{thm:eight-hole-nonzero-sum}]
If $s=4$, then $N=8$ and every request is a hole, so admissibility forces the
eight-hole sum to be zero.  Thus the nonzero-sum case is vacuous.

Assume $s\geq5$ and put
\[
 K=H,\qquad N=|K|=2^{s-1},\qquad r=s-1\geq4.
\]
Let the eight holes be $h_1,\ldots,h_8$.  As in the proof of
Theorem~\ref{thm:eight-hole-zero-sum}, the projected crossing profile
$\mathbf m$ has size $N-8$ and satisfies
\[
 \Sigma:=\sigma(\mathbf m)=h_1+\cdots+h_8\neq0.
\]

If no hole equals $\Sigma$, Proposition~\ref{prop:generic-eight} gives the
deleted boundaries.  If $\Sigma$ is a hole and a paired-lift datum exists,
Proposition~\ref{prop:paired-lift} gives them.  Otherwise
Lemma~\ref{lem:failure-classification} places the configuration in one of the
$23$ printed relation-space orbits.  When $r=4$, the unique rank-five
representative cannot occur because its hole span has dimension five; every
other printed representative is
realized in $\F_2^4$.  Thus Lemma~\ref{lem:certificate-transport} applies in
rank four, and it applies to all entries when $r\geq5$.  Lemmas
~\ref{lem:finite-exceptional} and~\ref{lem:strong-projector}, followed by
Proposition~\ref{prop:filtered-minor}, therefore give deleted sets $A,B$ with
\[
 [z^{\mathbf m}]\det C(z)_{A^c,B^c}\neq0.
\]
In every branch, \eqref{eq:minor-certificate} supplies the residual
bijection, and Lemma~\ref{lem:boundary-assembly} completes the solution.
\end{proof}

\begin{proof}[Proof of Theorem~\ref{thm:main}\textup{(ii)}]
If the eight holes have total sum zero, apply
Theorem~\ref{thm:eight-hole-zero-sum}.  Otherwise apply
Theorem~\ref{thm:eight-hole-nonzero-sum}.
\end{proof}

\begin{corollary}[Eight even-weight requests]\label{cor:eight-even}
Every admissible request instance in $\F_2^s$, $s\geq4$, with exactly eight
even-weight requests has a solution.
\end{corollary}

\begin{proof}
Use the parity hyperplane $H_{\rm even}$; its internal requests are exactly
the even-weight requests, and apply Theorem~\ref{thm:main}\textup{(ii)}.
\end{proof}

\subsection{Scope of the finite verification}
The theoretical reductions leave exactly three finite assertions, isolated
in Lemmas~\ref{lem:paired-local-signature},
\ref{lem:failure-classification}, and
\ref{lem:finite-exceptional}.  Table~\ref{tab:finite-verifications} records,
for each of these lemmas, the finite universe checked and the precise
conclusion used in the proof.  The appendix lists the $23$ exceptional
configurations and $45$ integer witnesses.  The exact source programs, inputs,
recorded outputs, checksum manifest, and reproduction instructions for these
three proof-critical verifications are archived in the immutable Zenodo
record~\cite{KreindelZabokritskiyVerification2026}; all arithmetic is exact
over $\F_2$ or the integers.

The point at which this method ceases to be uniform is now visible.  For the
two-direction filters used here, two independent directions generate a
four-element translation space $T$.  A four-element common support is one
$T$-coset, and an eight-element common support is two $T$-cosets, hence an
affine three-flat.  A sixteen-element support may be a union of four
$T$-cosets without being affine; moreover every such union has zero character
sum.  Thus neither general coefficient law alone controls the next
power-of-two level.  For an even number of holes that is not a power of two,
the clean block recursion and valuation laws proved above are not available
in the first place.  The Fourier reduction continues for every deletion size,
but these arithmetic and support classifications do not.

\appendix
\setcounter{table}{0}
\renewcommand{\thetable}{\thesection\arabic{table}}
\section{Exceptional eight-hole certificates}
\label{app:certificates}

The explicit data below are the finite input needed to complete the proof of
Lemma~\ref{lem:finite-exceptional}.  For each of the $23$ exceptional orbit
representatives, the list gives legal witnesses and the integer combination
of their local signatures required by Lemma~\ref{lem:strong-projector}.
The corresponding exact verifiers, inputs, and outputs are archived in the
Zenodo record~\cite{KreindelZabokritskiyVerification2026}.

\begin{table}[htbp]
\centering
\caption{The three exact finite verifications used in the unrestricted
eight-hole proof.}
\label{tab:finite-verifications}
\small
\begin{tabularx}{\textwidth}{@{}p{0.22\textwidth}XX@{}}
\toprule
check & finite universe & conclusion used in the proof \\
\midrule
local paired-filter signature
& $1024$ admissible basis patterns, $256$ external-sign patterns, and
$2{,}097{,}152$ common supports
& one surviving orbit coordinate, with coefficient of absolute value
$128$ or $384$ and valuation $7$ \\
relation-space classification
& all $417{,}199$ subspaces $\mathcal R$ of $\F_2^8$; $86{,}698$ satisfy
$e_i,\mathbf1\notin\mathcal R$ for every $i$ and
$\mathbf1+e_i\in\mathcal R$ for some $i$
& the $12{,}266$ paired-lift failures form exactly $23$ label-permutation
orbits \\
printed certificates
& $23$ entries and $45$ integer witnesses
& every legality, strong connectedness, support, valuation, and projector
identity required by the exceptional-orbit argument holds \\
\bottomrule
\end{tabularx}
\end{table}

\subsection{Certificate notation}

For a printed hole multiset $(h_1,\ldots,h_8)$, put
\[
 \Sigma=h_1+\cdots+h_8,
 \qquad
 m_{\rm tot}=|\{i:h_i=\Sigma\}|.
\]
Here and below all sums are in the relevant binary vector space.  The
representatives are grouped by the meaningful invariant $m_{\rm tot}$.  We
call the $j$th representative in the group with value $m$ \emph{Case
$m.j$}.  Thus the first index records the multiplicity of the hole sum, while
the second merely enumerates the representatives within that group.  Case
$3.4$ is the unique listed representative whose hole span has dimension
five.  Every other printed representative is realized in the ambient space
$\F_2^4$, although its hole span may have smaller dimension.

Integers are decimal encodings of binary vectors.  Thus
\[
 u=\sum_{j=0}^{d-1}u_j2^j
 \quad\hbox{encodes}\quad
 (u_0,\ldots,u_{d-1})\in\F_2^d,
\]
and vector addition is bitwise exclusive-or.  For example, in four
dimensions the hole value $9=1+8$ denotes $(1,0,0,1)$.  Hole values are
always ambient vectors written in standard coordinates.  Repeated numerical
values represent distinct labeled copies of the same hole vector.  In
particular, a certificate heading such as ``hole sum $=8$'' is shorthand for
a vector equation in $\F_2^d$: the vector
$\Sigma=h_1+\cdots+h_8$ has decimal code $8$.  Thus ``hole sum $=8$'' means
$\Sigma=(0,0,0,1)\in\F_2^4$, while the Case~$3.4$ heading
``hole sum $=16$'' means
$\Sigma=(0,0,0,0,1)\in\F_2^5$; neither is an ordinary integer sum.

The bracket
\[
c_j\,[(p_1,\ldots,p_4);(\alpha_1,\ldots,\alpha_4)
 \mid(q_1,\ldots,q_4);(\beta_1,\ldots,\beta_4)]
\]
encodes the witness
$\mathcal W_j=(P;\boldsymbol\alpha\mid R;\boldsymbol\beta)$ from
Section~\ref{sec:eight-sum-hole}.  The vertical bar separates the left
and right boundary filters; each semicolon separates an ordered hole tuple
from its ordered direction tuple.  Pairing is positionwise:
$\alpha_i$ filters $p_i$ and $\beta_i$ filters $q_i$.  The multiset union of
$P$ and $R$ is the displayed eight-hole multiset, and the order is part of
the witness because it fixes the wedge-factor order.  The printed integer
$c_j$ weights the local signature $\mathcal H_{\mathcal W_j}$ in
\[
 \sum_j c_j\mathcal H_{\mathcal W_j}
   =c_{\mathrm{proj}} e_{\mathcal O},
 \qquad c_{\mathrm{proj}}\neq0,
\]
where
$\mathcal O\in\mathfrak O_\Lambda=\binom{\Lambda}{8}/\Lambda$ is a free orbit
of eight-subsets under translation and $e_{\mathcal O}$ is the corresponding
standard coordinate vector in $\mathbb Z^{\mathfrak O_\Lambda}$.  The target
orbit may differ from one entry to another.

The printed direction integers are four-bit coordinate masks in the displayed
ordered basis $(\lambda_1,\lambda_2,\lambda_3,\lambda_4)$ of the local character space
$\Lambda$.  The mask $a=\sum_{j=0}^3a_j2^j$ denotes the actual character
$\widetilde a=\sum_{j=0}^3a_j\lambda_{j+1}$.  Character evaluation is the parity
of the common one-bits of $\widetilde a$ and the ambient hole value; its Walsh
sign is $\eps_{\widetilde a}(p)=(-1)^{\widetilde a(p)}$.  Thus the decoded
directions in every printed witness satisfy the legality conditions
\[
 \widetilde\alpha_i(p_i)=\widetilde\beta_i(q_i)=1
 \qquad(1\leq i\leq4).
\]
The verifier also checks
\[
 \sum_{i=1}^4\widetilde\alpha_i
 =\sum_{i=1}^4\widetilde\beta_i=\kappa,
 \qquad
 \kappa(\Sigma)=1,
\]
which is the character-sum condition used in the projector criterion.
Except for Case~$3.4$, the ordered basis is $(1,2,4,8)$, so a direction mask
happens to have the same numerical encoding as the actual character.  For
Case~$3.4$,
\[
 \Lambda=\langle3,4,8,16\rangle\leq(\F_2^5)^*.
\]
Consequently mask $1$ denotes the character $3$, mask $7$ denotes
$3+4+8=15$, and mask $9$ denotes $3+16=19$; these masks do not denote the
ambient vectors $1$, $7$, and $9$.

\subsection{Certificate list}

For each single-witness entry,
$\mathcal H_{\mathcal W}=c_{\mathrm{proj}} e_{\mathcal O}$ and the unique
nonzero coordinate has $2$-adic valuation
$\vTwo(c_{\mathrm{proj}})=7$, meaning
$2^7\mid c_{\mathrm{proj}}$ but $2^8\nmid c_{\mathrm{proj}}$.  The four
multi-witness entries, Cases~$2.3$, $2.5$, $4.2$, and $6.1$, instead print all
coefficients $c_j$ and then state the resulting projector scalar
$c_{\mathrm{proj}}$
explicitly.  Throughout the list, numbers inside hole multisets and witness
brackets, as well as the displayed hole sums, are decimal codes of binary
vectors or characters; only the coefficients $c_j$ and projector scalars
$c_{\mathrm{proj}}$ are ordinary integers.  The exact target coordinates are
reconstructed from the printed witnesses and recorded in the verification
output archived in the Zenodo
record~\cite{KreindelZabokritskiyVerification2026}; the proof uses only that
the resulting coordinate is a free translation orbit and that its coefficient
is nonzero.

\medskip
\noindent\textbf{Cases with $m_{\rm tot}=1$.}

\paragraph{Case 1.1.} Hole multiset $\{1,1,2,4,6,8,9,9\}$; hole sum $=8$; local character basis $(1,2,4,8)$.
\begin{align*}
&1\,[(1,1,2,4);(1,1,14,6)\mid (6,8,9,9);(11,10,3,10)]
\end{align*}

\paragraph{Case 1.2.} Hole multiset $\{1,1,1,2,2,2,3,8\}$; hole sum $=8$; local character basis $(1,2,4,8)$.
\begin{align*}
&1\,[(1,1,1,2);(1,9,15,15)\mid (2,2,3,8);(3,11,14,14)]
\end{align*}

\paragraph{Case 1.3.} Hole multiset $\{1,1,1,2,3,8,9,9\}$; hole sum $=8$; local character basis $(1,2,4,8)$.
\begin{align*}
&1\,[(1,1,1,2);(1,9,7,7)\mid (3,8,9,9);(14,11,3,14)]
\end{align*}

\paragraph{Case 1.4.} Hole multiset $\{1,2,3,4,5,6,7,8\}$; hole sum $=8$; local character basis $(1,2,4,8)$.
\begin{align*}
&1\,[(1,2,3,4);(1,6,1,14)\mid (5,6,7,8);(4,12,15,15)]
\end{align*}

\paragraph{Case 1.5.} Hole multiset $\{1,2,3,8,11,11,11,11\}$; hole sum $=8$; local character basis $(1,2,4,8)$.
\begin{align*}
&1\,[(1,2,3,8);(1,3,1,10)\mid (11,11,11,11);(2,2,5,12)]
\end{align*}

\paragraph{Case 1.6.} Hole multiset $\{1,1,1,1,2,4,6,8\}$; hole sum $=8$; local character basis $(1,2,4,8)$.
\begin{align*}
&1\,[(1,1,1,1);(1,7,9,7)\mid (2,4,6,8);(2,13,13,10)]
\end{align*}

\paragraph{Case 1.7.} Hole multiset $\{1,1,1,1,1,2,3,8\}$; hole sum $=8$; local character basis $(1,2,4,8)$.
\begin{align*}
&1\,[(1,1,1,1);(1,1,15,7)\mid (1,2,3,8);(11,11,2,10)]
\end{align*}

\medskip
\noindent\textbf{Cases with $m_{\rm tot}=2$.}

\paragraph{Case 2.1.} Hole multiset $\{1,2,2,8,8,9,10,10\}$; hole sum $=8$; local character basis $(1,2,4,8)$.
\begin{align*}
&1\,[(1,2,2,8);(1,7,7,9)\mid (8,9,10,10);(13,10,2,13)]
\end{align*}

\paragraph{Case 2.2.} Hole multiset $\{1,1,1,1,2,8,8,10\}$; hole sum $=8$; local character basis $(1,2,4,8)$.
\begin{align*}
&1\,[(1,1,1,2);(1,9,7,7)\mid (1,8,8,10);(11,14,14,3)]
\end{align*}

\paragraph{Case 2.3.} Hole multiset $\{1,1,1,8,8,9,9,9\}$; hole sum $=8$; local character basis $(1,2,4,8)$.
\begin{align*}
&1\,[(1,1,1,8);(1,1,1,9)\mid (8,9,9,9);(11,5,8,14)]\\
&1\,[(1,1,1,8);(1,1,1,9)\mid (8,9,9,9);(11,7,8,12)]\\
&-3\,[(1,1,1,8);(1,1,1,9)\mid (8,9,9,9);(13,3,8,14)]\\
&1\,[(1,1,1,8);(1,1,1,9)\mid (8,9,9,9);(13,7,8,10)]\\
&1\,[(1,1,1,8);(1,1,1,9)\mid (8,9,9,9);(15,3,8,12)]\\
&1\,[(1,1,1,8);(1,1,1,9)\mid (8,9,9,9);(15,5,8,10)]
\end{align*}
The exact combined local orbit signature is $768e_{\mathcal O}$.

\paragraph{Case 2.4.} Hole multiset $\{1,2,3,8,8,9,10,11\}$; hole sum $=8$; local character basis $(1,2,4,8)$.
\begin{align*}
&1\,[(1,2,3,8);(1,6,6,8)\mid (8,9,10,11);(10,5,3,5)]
\end{align*}

\paragraph{Case 2.5.} Hole multiset $\{1,1,1,1,1,8,8,9\}$; hole sum $=8$; local character basis $(1,2,4,8)$.
\begin{align*}
&-1\,[(1,1,1,1);(1,1,1,9)\mid (1,8,8,9);(3,8,13,14)]\\
&-1\,[(1,1,1,1);(1,1,1,9)\mid (1,8,8,9);(11,8,12,7)]\\
&1\,[(1,1,1,1);(1,1,1,9)\mid (1,8,8,9);(13,8,14,3)]\\
&-1\,[(1,1,1,1);(1,1,1,9)\mid (1,8,8,9);(15,8,10,5)]
\end{align*}
The exact combined local orbit signature is $384e_{\mathcal O}$.

\paragraph{Case 2.6.} Hole multiset $\{1,2,4,8,8,9,10,12\}$; hole sum $=8$; local character basis $(1,2,4,8)$.
\begin{align*}
&1\,[(1,2,4,8);(3,7,7,11)\mid (8,9,10,12);(9,14,6,9)]
\end{align*}

\medskip
\noindent\textbf{Cases with $m_{\rm tot}=3$.}

\paragraph{Case 3.1.} Hole multiset $\{1,1,1,2,3,8,8,8\}$; hole sum $=8$; local character basis $(1,2,4,8)$.
\begin{align*}
&1\,[(2,3,8,8);(2,10,11,10)\mid (1,1,1,8);(1,1,7,14)]
\end{align*}

\paragraph{Case 3.2.} Hole multiset $\{1,1,2,4,6,8,8,8\}$; hole sum $=8$; local character basis $(1,2,4,8)$.
\begin{align*}
&1\,[(1,1,6,8);(1,5,5,9)\mid (2,4,8,8);(3,12,11,12)]
\end{align*}

\paragraph{Case 3.3.} Hole multiset $\{1,1,2,8,8,8,9,11\}$; hole sum $=8$; local character basis $(1,2,4,8)$.
\begin{align*}
&1\,[(1,2,8,9);(3,10,11,10)\mid (1,8,8,11);(5,13,15,15)]
\end{align*}

\paragraph{Case 3.4.} Hole multiset $\{1,2,4,8,15,16,16,16\}$; hole sum $=16$; local character basis $(3,4,8,16)$.
\begin{align*}
&1\,[(1,2,8,16);(1,7,7,9)\mid (4,15,16,16);(2,10,14,14)]
\end{align*}

\paragraph{Case 3.5.} Hole multiset $\{1,2,4,8,8,8,9,14\}$; hole sum $=8$; local character basis $(1,2,4,8)$.
\begin{align*}
&1\,[(2,8,8,14);(2,15,10,15)\mid (1,4,8,9);(1,5,9,5)]
\end{align*}

\medskip
\noindent\textbf{Cases with $m_{\rm tot}=4$.}

\paragraph{Case 4.1.} Hole multiset $\{1,2,2,8,8,8,8,9\}$; hole sum $=8$; local character basis $(1,2,4,8)$.
\begin{align*}
&1\,[(1,2,2,8);(3,2,3,10)\mid (8,8,8,9);(9,9,15,7)]
\end{align*}

\paragraph{Case 4.2.} Hole multiset $\{1,1,1,8,8,8,8,9\}$; hole sum $=8$; local character basis $(1,2,4,8)$.
\begin{align*}
&-1\,[(1,1,1,8);(1,1,1,9)\mid (8,8,8,9);(8,10,13,7)]\\
&-1\,[(1,1,1,8);(1,1,1,9)\mid (8,8,8,9);(8,10,15,5)]\\
&-1\,[(1,1,1,8);(1,1,1,9)\mid (8,8,8,9);(8,11,12,7)]\\
&-1\,[(1,1,1,8);(1,1,1,9)\mid (8,8,8,9);(8,11,14,5)]\\
&-1\,[(1,1,1,8);(1,1,1,9)\mid (8,8,8,9);(8,12,15,3)]\\
&3\,[(1,1,1,8);(1,1,1,9)\mid (8,8,8,9);(8,13,14,3)]
\end{align*}
The exact combined local orbit signature is $768e_{\mathcal O}$.

\paragraph{Case 4.3.} Hole multiset $\{1,2,4,8,8,8,8,15\}$; hole sum $=8$; local character basis $(1,2,4,8)$.
\begin{align*}
&1\,[(1,2,4,8);(3,14,14,11)\mid (8,8,8,15);(9,13,13,1)]
\end{align*}

\medskip
\noindent\textbf{Cases with $m_{\rm tot}=5$.}

\paragraph{Case 5.1.} Hole multiset $\{1,2,3,8,8,8,8,8\}$; hole sum $=8$; local character basis $(1,2,4,8)$.
\begin{align*}
&1\,[(1,2,8,8);(11,3,11,10)\mid (3,8,8,8);(6,8,8,15)]
\end{align*}

\medskip
\noindent\textbf{Cases with $m_{\rm tot}=6$.}

\begin{samepage}
\allowdisplaybreaks[0]
\paragraph{Case 6.1.} Hole multiset $\{1,8,8,8,8,8,8,9\}$; hole sum $=8$; local character basis $(1,2,4,8)$.
\begin{align*}
&-21\,[(1,8,8,8);(1,8,10,11)\mid (8,8,8,9);(9,9,13,5)]\\
&15\,[(1,8,8,8);(1,8,10,11)\mid (8,8,8,9);(9,9,15,7)]\\
&21\,[(1,8,8,8);(1,8,12,13)\mid (8,8,8,9);(9,9,11,3)]\\
&36\,[(1,8,8,8);(1,8,12,13)\mid (8,8,8,9);(9,9,15,7)]\\
&-15\,[(1,8,8,8);(1,8,14,15)\mid (8,8,8,9);(9,9,11,3)]\\
&-36\,[(1,8,8,8);(1,8,14,15)\mid (8,8,8,9);(9,9,13,5)]\\
&-32\,[(1,8,8,8);(1,9,9,9)\mid (8,8,8,9);(8,10,13,7)]\\
&28\,[(1,8,8,8);(1,9,9,9)\mid (8,8,8,9);(8,10,15,5)]\\
&-25\,[(1,8,8,8);(1,9,9,9)\mid (8,8,8,9);(8,11,12,7)]\\
&23\,[(1,8,8,8);(1,9,9,9)\mid (8,8,8,9);(8,11,14,5)]
\end{align*}
The exact combined local orbit signature is $2304e_{\mathcal O}$.
\end{samepage}

The $23$ entries above contain $45$ witnesses.  The exact verifier archived in
the Zenodo record~\cite{KreindelZabokritskiyVerification2026} expands
them over the integers and checks legality, strong
$\Lambda$-connectedness, the common-support character-sum condition, and the
stated coordinate-projector identities.  Hence every hypothesis of
Lemma~\ref{lem:strong-projector} holds for each exceptional representative,
completing the finite assertion in Lemma~\ref{lem:finite-exceptional}.

\end{document}